\documentclass[11pt]{amsart}

\usepackage{stmaryrd} 
\usepackage{amssymb}
\usepackage{mathabx} 
\usepackage{amsmath} 
\usepackage{amscd}
\usepackage{amsbsy}
\usepackage{leftindex}
\usepackage{comment,enumerate}
\usepackage[matrix,arrow,cmtip]{xy}
\usepackage[pagebackref]{hyperref}
\hypersetup{
 colorlinks=true,
 linkcolor=violet,
 filecolor=magenta,
 citecolor=brown,
 urlcolor=teal,
 pdftitle={Euler Systems KRV},
 pdfpagemode=FullScreen, }
\usepackage{mathrsfs}
\usepackage{color}
\usepackage{mathtools,caption}
\usepackage{tikz-cd}
\usepackage{longtable}
\usepackage[utf8]{inputenc}
\usepackage[OT2,T1]{fontenc}
\usepackage{changepage}
\usepackage{mathtools}
\usepackage{commath}
\usepackage[left=2.7cm,right=2.7cm,top=3.5cm,bottom=3cm]{geometry}

\numberwithin{equation}{section}

\usepackage{graphicx}
\newcommand{\Bmu}{\mbox{$\raisebox{-0.59ex}
  {$l$}\hspace{-0.18em}\mu\hspace{-0.88em}\raisebox{-0.98ex}{\scalebox{2}
  {$\color{white}.$}}\hspace{-0.416em}\raisebox{+0.88ex}
  {$\color{white}.$}\hspace{0.46em}$}{}}

\newcounter{lettera}

\newtheorem{itheorem}[lettera]{Theorem}  

\DeclareMathOperator{\Gr}{Gr}

\DeclareMathOperator{\Spec}{Spec}

\DeclareMathOperator{\Norm}{norm}
\DeclareMathOperator{\pr}{pr}
\DeclareMathOperator{\cris}{cris}

\DeclareMathOperator{\Diag}{Diag}
\DeclareMathOperator{\Hom}{Hom}

\DeclareMathOperator{\cores}{cores}
\DeclareMathOperator{\Gal}{Gal}

\DeclareMathOperator{\ord}{ord}

\DeclareMathOperator{\charac}{char}

\DeclareMathOperator{\tor}{tor}

\DeclareMathOperator{\coker}{coker}
\DeclareMathOperator{\GL}{GL}
\DeclareMathOperator{\SL}{SL}
\DeclareMathOperator{\cyc}{cyc}

\DeclareMathOperator{\cont}{cont}

\DeclareMathOperator{\Sym}{Sym}
\DeclareMathOperator{\TSym}{TSym}
\DeclareMathOperator{\mom}{mom}
\DeclareMathOperator{\Sp}{sp}

\DeclareMathOperator{\Fil}{Fil}
\DeclareMathOperator{\loc}{loc}
\DeclareMathOperator{\Iw}{Iw}

\newcommand{\sF}{\mathscr{F}}
\newcommand{\sH}{\mathscr{H}}
\newcommand{\sP}{\mathscr{P}}
\newcommand{\sR}{\mathscr{R}}
\newcommand{\sS}{\mathscr{S}}
\newcommand{\Q}{\mathbb Q}
\newcommand{\D}{\mathbb D}
\newcommand{\T}{\mathbb T}
\newcommand{\F}{\mathbb F}
\newcommand{\Z}{\mathbb Z}
\newcommand{\C}{\mathbb C}

\newcommand{\mbA}{\mathbf A}
\newcommand{\mbf}{\mathbf f}
\newcommand{\mbg}{\mathbf g}
\newcommand{\mbj}{\mathbf j}
\newcommand{\mba}{\mathbf a}
\newcommand{\mbk}{\mathbf k}

\newcommand{\mbT}{\mathbf T}
\newcommand{\mbV}{\mathbf V}
\newcommand{\cO}{\mathcal O}
\newcommand{\fP}{\mathfrak{P}}

\newcommand{\fm}{\mathfrak{m}}

\newcommand{\Zp}{\Z_p}
\newcommand{\bPsi}{\mathbf{\Psi}}
\newcommand{\bkappa}{\boldsymbol{\kappa}}

\newcommand{\Exterior}{\mathchoice{{\textstyle\bigwedge}}%
    {{\bigwedge}}%
    {{\textstyle\wedge}}%
    {{\scriptstyle\wedge}}}
    
\newcommand{\defeq}{\vcentcolon=}

\newcommand{\longepi}{\mbox{\;$\relbar\joinrel\twoheadrightarrow$\;}}

\DeclareSymbolFont{cyrletters}{OT2}{wncyr}{m}{n}
\DeclareMathSymbol{\Sha}{\mathalpha}{cyrletters}{"58}
\DeclareMathSymbol{\Zh}{\mathalpha}{cyrletters}{"11}

\usepackage[OT2,T1]{fontenc}
\input{cyracc.def}

\newcounter{para}

\newtheorem*{Theorem*}{Theorem}
\newtheorem*{Conj*}{Conjecture}
\newtheorem*{Defi*}{Definition}
\newtheorem*{hypo*}{Hypothesis}
\newtheorem*{prop*}{Proposition}
\newtheorem{Th}{Theorem}[section]
\newtheorem{Lemma}[Th]{Lemma}

\newtheorem{prop}[Th]{Proposition}

\theoremstyle{definition}
\newtheorem{Defi}[Th]{Definition}
\theoremstyle{remark}
\newtheorem{Remark}[Th]{Remark}
\newtheorem*{Remark*}{Remark}

\definecolor{cadmiumgreen}{rgb}{0.0, 0.42, 0.24}

\begin{document}
\title[Euler systems and the symmetric square of a Hida family]{Euler systems and the symmetric square\\of a Hida family}

\author{Debanjana Kundu, Jishnu Ray and Stefano Vigni}

\address[Kundu]{Department of Mathematics and Statistics, University of Regina, 3737 Wascana Parkway, Regina, Saskatchewan, S4S 0A2 Canada}
\email{debanjana.kundu@uregina.ca}

\address[Ray]{Harish-Chandra Research Institute, Chhatnag Road, Jhunsi, Prayagraj 211 019, India}
\address[Ray]{Homi Bhabha National Institute, Training School Complex, Anushakti Nagar, Mumbai 400 094, India}
\email{jishnuray@hri.res.in}

\address[Vigni]{Dipartimento di Matematica, Universit\`a di Genova, Via Dodecaneso 35, 16146 Genova, Italy}
\email{stefano.vigni@unige.it}

\thanks{During the initial stages of the project, the first author was supported by a PIMS postdoctoral fellowship; she is currently supported by an NSERC Discovery grant RGPIN-2026-07384. The second author gratefully acknowledges support from Inspire Research Grant, Department of Science and Technology, Govt. of India. The third author was partially supported by PRIN 2022 ``The arithmetic of motives and $L$-functions'' and by the GNSAGA group of INdAM. The research by the third author is also partially supported by the MUR Excellence Department Project awarded to Dipartimento di Matematica, Universit\`a di Genova, CUP D33C23001110001}



\begin{abstract} 
Let $p\geq7$ be a prime number. We build a non-trivial Euler system for the symmetric square of a $p$-adic Hida family of modular forms interpolating the Euler system constructed by Loeffler--Zerbes for the symmetric square of a $p$-ordinary newform. As a second contribution, we prove an algebraic functional equation for dual Selmer groups in this setting. Finally, building on recent work by B\"uy\"ukboduk--Ganguly on functional equations of algebraic (Rankin--Selberg) $p$-adic $L$-functions, we prove a divisibility result towards the Iwasawa main conjecture for the symmetric square of a Hida family. 
\end{abstract}

\maketitle

\section{Introduction} \label{section: intro}

\subsection{Background and motivation}

The Iwasawa theory of Galois representations attached to modular forms has been profoundly enriched in recent years by the systematic construction of Euler systems. Let $p\geq7$  be a prime number. For a $p$-ordinary eigenform $f$, denote by $M(f)$ the two-dimensional $p$-adic Galois representation attached to $f$ by Deligne (\cite{Del-Bourbaki}). The Iwasawa main conjecture predicts a deep relationship between the characteristic ideal of the Selmer group of $M(f)$ and a $p$-adic $L$-function interpolating critical $L$-values. While the classical case of this two-dimensional representation was largely settled by Kato (\cite{Kato}), the arithmetic of its three-dimensional symmetric square $\Sym^2 M(f)$, which is intimately linked to the deformation theory of $f$ and modularity lifting theorems, has historically proved much more resistant.

A major breakthrough was achieved in \cite{LZ19} by Loeffler and Zerbes, who built a non-trivial Euler system for the symmetric square of a $p$-ordinary eigenform. By projecting the Beilinson--Flach classes attached to the Rankin--Selberg convolution $f\otimes f$ to the symmetric square direct summand, they successfully bounded the associated Selmer groups and deduced a divisibility a result in the direction of the relevant Iwasawa main conjecture. However, passing from a single modular form to a $p$-adic family of modular forms \emph{\`a la} Hida (\cite{hida86b, hida86a}) introduces profound new arithmetic and analytic phenomena. The goal of this article is to generalize the Iwasawa theory of the symmetric square of a modular form to the setting of $p$-adic Hida families. 

More precisely, let $\mbf$ be a $p$-adic Hida family of tame level $N_\mbf$ and let $M(\mbf)^*$ be its associated ``big'' Galois representation over Hida's Hecke algebra. We construct a non-trivial Euler system for the symmetric square representation $\Sym^2M(\mathbf{f})^*$, establish an algebraic functional equation for the associated Selmer groups and ultimately prove a divisibility result towards the three-variable Iwasawa main conjecture for the symmetric square $\Sym^2\mbf$ of $\mbf$.

\subsection{Euler systems for Hida families} 

Our first main contribution is the explicit construction of an Euler system over the cyclotomic $\Zp$-extension. Following the philosophy of Mazur--Tate in \cite{MazurTate87}, we consider the Beilinson--Flach classes $\leftindex_c{\mathcal{BF}}_m^{\mbf,\mbf}$ constructed by Kings, Loeffler and Zerbes in \cite{KLZ17} for the (completed) tensor product $M(\mathbf{f})^* \hat{\otimes}M(\mathbf{f})^*$. When $\mathbf{f}$ is non-Eisenstein and $p$-distinguished, this rank-$4$ lattice decomposes into its symmetric and alternating components. By applying appropriate involutions and twisting by a Dirichlet character $\psi$, whose conductor we denote by $N_\psi$, we can isolate classes taking values in the rank-$3$ lattice $\mbT:=\Sym^2 M(\mathbf{f})^*(1)(\psi)$. 

In the following statement, using self-explanatory notation for roots of unity, let us write $H^1_{\Iw}\bigl(\Q(\Bmu_{mp^\infty}),\mbT\bigr)$ for the first Iwasawa cohomology group with values in $\mbT$ and let $\nu\defeq\psi(-1)$ be the parity of $\psi$. Moreover, denote by $\sR$ the set of square-free integers $m\geq1$ such that $(m,6pcN_\mbf N_\psi)=1$, where $c>1$ is an auxiliary integer we eventually eliminate (\emph{cf.} \S\ref{c-subsec}). 

\begin{itheorem}[Theorem~\ref{LZ 4.1.6}] \label{A-intro}
There exists a non-trivial Euler system $\mathbf{c}={(c_m)}_{m\in\sR}$ of cohomology classes $c_m \in H^1_{\Iw}\bigl(\Q(\Bmu_{mp^\infty}),\mbT\bigr)^\nu$ interpolating the Loeffler--Zerbes Euler system for each classical specialization.
\end{itheorem}

In order to extract bounds on the Selmer groups from this Euler system, we generalize the Kolyvagin system machinery to the setting of Hida families over multi-variable Iwasawa algebras. This requires us to establish a strong ``big image'' result (Theorem~\ref{thm: big image}) for the residual Galois representation, which ensures that the Kolyvagin derivatives satisfy the strict Greenberg local conditions at $p$.

\subsection{Parity obstruction and Selmer complexes}

While Theorem~\ref{A-intro} offers an algebraic bound on (dual) Selmer groups, connecting this bound with the analytic $p$-adic $L$-function presents an essential parity obstruction. More precisely, the Kolyvagin system produces a divisibility only for the $\eta$-isotypic component of the Selmer group, where $\eta$ is a character of the torsion subgroup of the cyclotomic Galois group satisfying the algebraic parity condition $\eta(-1) = \psi(-1)$. However, Hida's three-variable $p$-adic $L$-function attached to $\Sym^2\mbf$ (\cite{Hida-fourier}) and Dasgupta's formula decomposing the Rankin--Selberg $p$-adic $L$-function into its symmetric square and Kubota--Leopoldt components (\cite{Das16}) are naturally defined only on the opposite half of the weight space, where $\psi(-1)=-\sigma(-1)$. To bridge this gap, we must transport the Euler system bound across the weight space via an algebraic functional equation. While in \S\ref{alg-subsec} we provide (under mild technical assumptions) a descent proof of this functional equation along classical lines, the proof of our divisibility result towards the Iwasawa main conjecture requires perfectness and amplitude conditions that classical Selmer groups do not possess. Thus, in \S\ref{div-sec}, we lift the arithmetic to the derived category, employing the Iwasawa-theoretic Selmer complexes (in the sense of Nekov\'a\v{r}, \cite{Nek-selmer}) recently developed by B\"uy\"ukboduk and Ganguly (\cite{BG}).

Henceforth, let $\Q_{\cyc}$ be the $\Zp$-extension of $\Q$, set $\Gamma_1\defeq\Gal(\Q_{\cyc}/\Q)$ and put $\mbT_\mbf\defeq\Sym^2 M(\mbf)^*$; moreover, $\Lambda_\mbf$ is a suitable (localized) Hecke algebra attached to $\mbf$, $X(\mbT_{\mathbf{f}}/\Q_{\cyc})$ is the Pontryagin dual of the Selmer group of $\mbT_{\mbf}\otimes_{\Lambda_{\mbf}}\Hom_{\cont}(\Lambda_{\mbf},\Q_p/\Zp)$ over $\Q_{\cyc}$ introduced in \S\ref{selmer-subsec} and $\iota$ denotes the Iwasawa involution.

\begin{itheorem}[Theorem~\ref{func-thm}] \label{B-intro}
There is an equality of characteristic ideals
\[
\charac_{\Lambda_\mbf[\![\Gamma_1]\!]}\bigl(X(\mbT_{\mathbf{f}}/\Q_{\cyc})\bigr)= \charac_{\Lambda_\mbf[\![\Gamma_1]\!]}\bigl(X(\mbT_\mbf^*/\Q_{\cyc})^\iota\bigr).
\]
\end{itheorem}

This result is a consequence of the fact that, from a somewhat more general perspective, the degree-$2$ cohomology of the Selmer complex for the symmetric square representation satisfies an algebraic functional equation under $\iota$.

\subsection{Main divisibility and reducibility obstruction}

By invoking Theorem~\ref{B-intro}, we transport the characteristic ideal bound from the $\eta$-branch to the dual $\sigma$-branch. On this half of the weight space, the explicit reciprocity law for Beilinson–Flach families maps the bottom class of our Euler system to the Rankin--Selberg $p$-adic $L$-function, which allows us to apply Dasgupta’s analytic factorization. Let $L_p(\Sym^2\mbf,\psi,\kappa,\sigma)$ be Hida's three-variable $p$-adic $L$-function attached to $\Sym^2\mbf$ and let $L_p(\psi\epsilon_\mbf,s-k+1)$ be the Kubota--Leopoldt $p$-adic $L$-function associated with the Dirichlet character $\psi\epsilon_\mbf$, where $\varepsilon_\mbf$ denotes the prime-to-$p$ part of the Nebentypus of $\mbf$.

This yields our final result towards the Iwasawa main conjecture, which we can state as follows.

\begin{itheorem}[Theorem~\ref{main-thm}] \label{C-intro}
The product of $L_p(\Sym^2\mbf,\psi,\kappa,\sigma)$ and $L_p(\psi\epsilon_\mbf,s-k+1)$ lies in the characteristic ideal of the $\sigma$-component of the dual Selmer group $X(\mbT_\mbf/\Q_{\cyc})$.
\end{itheorem}

It is worth remarking that the presence of the Kubota--Leopoldt factor in Theorem~\ref{C-intro} does not reflect a weakness of the explicit reciprocity law; rather, it is a manifestation of an apparently unavoidable structural limitation in the Kolyvagin system machinery, which might be referred to as a ``reducibility obstruction''. Since the Rankin--Selberg representation is reducible modulo $p$, it fails the basic hypotheses required by the Mazur--Rubin theory (\cite{MR_KS}). We are thus forced to project the Beilinson--Flach classes to the irreducible symmetric square lattice before generating the Kolyvagin system. This yields a bound for $X(\mbT_\mbf/\Q_{\cyc})$, but prevents the cancellation of the undesired Dirichlet $L$-function term.

\subsection{Future applications}

Beyond the Iwasawa main conjecture, the Euler system constructed herein paves the way for deeper arithmetic investigations of Hida families. In particular, we anticipate that these classes can be utilized to construct theta elements \emph{\`a la} Mazur--Tate (\cite{MazurTate87}) and control Fitting ideals of Selmer groups at finite layers, generalizing conjectures of Kurihara to the higher-rank setting (see, \emph{e.g.}, \cite{kurihara-inv, kurihara-crelle, kurihara-PLMS, kurihara-munster}). We plan to investigate these questions in a future project.

\subsection{Outline of the paper}

In \S\ref{section: preliminaries}, we review the construction of the Beilinson--Flach classes for the Rankin--Selberg convolution of two Hida families. In \S\ref{S: results 1}, we perform the projection and the twist to construct the Euler system for $\Sym^2M(\mbf)^*$. In \S\ref{sec:Koly}, we prove the necessary big image results and obtain a bound on (dual) Selmer groups. We prove the algebraic functional equation via classical descent arguments in \S\ref{sec:AFE}. Finally, in \S\ref{div-sec}, we introduce the Selmer complexes, establish the derived functional equation and deduce our divisibility result towards the Iwasawa main conjecture.

\subsection*{Acknowledgements}
It is a pleasure to thank Kazim B\"uy\"ukboduk, Samit Dasgupta, Antonio Lei and David Loeffler for discussions and exchanges on some of the topics of this paper. 

\section{Preliminaries} \label{section: preliminaries}

In this section, we collect background material and preliminaries on modular curves, modular forms and their associated Galois representations.

\subsection{Modular curves}

Choose integers $M,N\geq 1$ such that $M+N\geq 5$ and $M|N$. For any $\Z[1/MN]$-scheme $\sS$, let $E/\sS$ be an elliptic curve and take $\sS$-valued points $e_1, e_2\in E(\sS)$ of order $M$ and $N$, respectively, with the property that the map 
\[ \Z/M\Z\times\Z/N\Z\longrightarrow E,\quad(m,n)\longmapsto me_1+ne_2 \]
is injective; there is an obvious notion of isomorphism of triples $(E/\sS,e_1,e_2)$ as above. The modular curve $Y(M,N)$ is the $\Z[1/MN]$-scheme representing the functor given on objects by
\[
\sS \longmapsto\bigl\{ \textrm{isomorphism classes of triples }(E/\sS,e_1,e_2)\bigr\}.
\]
As customary, we shall denote $Y(1,N)$ by $Y_1(N)$. Moreover, if we set
\[
\Gamma(M,N) \defeq \Bigl\{g\in \SL_2(\Z)\;\,\Big|\;\, g\equiv 1 \pmod{\bigl(\begin{smallmatrix}M & M\\ N & N\end{smallmatrix}\bigr)}\Bigr\}, 
\]
then there is an analytic uniformization
\[
Y(M,N)(\C) \simeq(\Z/M\Z)^\times \times \Gamma(M,N)\backslash \mathcal{H},
\]
where $\mathcal{H}\defeq\bigl\{z\in\C\mid\Im(z)>0\bigr\}$ is the complex upper half-plane. See, \emph{e.g.}, \cite[\S2.1]{Kato} and \cite[\S2.3]{KLZ17} for details.

\subsection{Coefficient sheaves} \label{section: coeff sheaves}

Let $p$ be a prime number. Write $\pi:\mathcal{E} \rightarrow Y_1(N)$ for the universal elliptic curve over $Y_1(N)$. Following \cite[p. 17]{KLZ17}, we consider the sheaf
\[
\sH_{\Q_p}\defeq\Bigl(R^1\pi_{*}\Q_p\Bigr)^{\!\vee}\simeq R^1\pi_* \Q_p(1)
\]
on $Y_1(N)[1/p]$, where we view $\Q_p$ as a constant sheaf on $\mathcal E$. Analogously, one can define $\sH_{\Zp}$ with $\Q_p$ replaced by $\Zp$: this is an \'etale $\Z_p$-sheaf of rank $2$ on $Y_1(N)[1/p]$. For any abelian group $H$, we define $\Sym^k H$ and $\TSym^k H$ as in \cite[\S2.2]{KLZ17}. Given an $h\in H$, we write
\[
h^{[k]}\defeq h^{\otimes k}\in\TSym^k H
\]
for the corresponding element of $\TSym^k H$. If $H$ is free of finite rank and $\{e_1, \ldots, e_d\}$ is a basis for $H$ over $\Z$, then a basis for $\TSym^k H$ over $\Z$ is given by
\[
\biggl\{e_1^{[n_1]}e_2^{[n_2]} \cdots e_d^{[n_d]} \;\Big|\; \sum_{i=1}^d n_i =k\biggr\}.
\]
There exists a natural ring homomorphism
\[ \Sym^k H\longrightarrow \TSym^k H,\quad e_1^{n_1} \cdots e_d^{n_d}\longmapsto k! e_1^{[n_1]} \cdots e_d^{[n_d]}
\]
that extends to a ring homomorphism 
\[ \mathrm{Sym}^\bullet H \longrightarrow \TSym^\bullet H, \]
where the direct sums $\Sym^\bullet H\defeq\oplus_{k \geq 0}\Sym^\mathrm{k} H $ and $\TSym^\bullet H\defeq\oplus_{k \geq 0}\TSym^\mathrm{k}H $ are equipped with natural ring structures. This homomorphism becomes an isomorphism in degrees up to $k$ after inverting $k!$.

For a regular $\Z[1/p]$-scheme $X$ and a locally constant \'etale sheaf $\mathcal{F}$ of $\Z/p^n\Z$-modules on $X$, we can similarly define an \'etale sheaf $\TSym^k\mathcal{F}$. This is a sheaf of symmetric tensors of degree $k$ over $\mathcal{F}$ that is isomorphic after inverting $k!$ to the $k$-th symmetric power.

\subsection{Galois representations of modular forms} \label{section: Galois rep}

Let $f\in S_{k+2}(\Gamma_0(N_f))$ be a normalized Hecke eigenform of weight $k+2\geq2$ and level $\Gamma_0(N_f)$, whose $q$-expansion is denoted by $f(q)=\sum_{n\geq1}a_n(f)q^n$. Let $L\defeq\Q\bigl(a_n(f)\mid n\geq1\bigr)$ be the Hecke field of $f$, \emph{i.e.}, the (totally real) number field generated over $\Q$ by the Fourier coefficients of $f$. Recall the sheaves $\sH_{\Q_p}$ and $\sH_{\Z_p}$ that were introduced in \S\ref{section: coeff sheaves}. For each prime $\fP$ of $L$ above $p$, write $L_\fP$ for the completion of $L$ at $\fP$ and denote by $M_{L_{\fP}}(f)$ the maximal subspace of the compactly supported cohomology
\[
H^1_{\text{\'et},c}\Bigl(Y_1(N_f)_{\overline{\Q}},\Sym^k\sH_{\Q_p}^\vee\Bigr)\otimes_{\,\Q_p}L_{\fP} 
\]
on which the Hecke operators $T_{\ell}$ for $\ell\nmid N_f$ and $U_{\ell}$ for $\ell\,|\,N_f$ act as multiplication by $a_{\ell}(f)$.

It is well known that $M_{L_{\fP}}(f)$ is a $2$-dimensional $L_{\fP}$-vector space equipped with a continuous action of the absolute Galois group $G_{\Q}\defeq\Gal(\overline\Q/\Q)$. The Galois representation $M_{L_{\fP}}(f)$ is unramified outside the finite set of places of $L$ dividing $pN_f\infty$. The dual module $M_{L_{\fP}}(f)^*$, endowed with the contragredient action of $G_\Q$, can be identified with the maximal quotient of
\[
H^1_{\text{\'et}}\Bigl(Y_1(N_f)_{\overline{\Q}},\TSym^k(\sH_{\Q_p})(1)\Bigr)\otimes_{\,\Q_p} L_{\fP}
\]
on which the dual Hecke operators $T'_{\ell}$ and $U'_{\ell}$ (see \cite[Definition~2.4.3 and \S2.8]{KLZ20}) act as multiplication by $a_{\ell}(f)$. Finally, let $\cO_{\fP}$ be the valuation ring of $L_{\fP}$ and write $M_{\cO_{\fP}}(f)^*$ for the $\cO_{\fP}$-lattice in $M_{L_{\fP}}(f)^*$ generated by the image of the integral \'etale cohomology group
\[
H^1_{\text{\'et}}\Bigl(Y_1(N_f)_{\overline{\Q}},\TSym^k(\sH_{\Z_p})(1) \Bigr)\otimes_{\,\Zp}\cO_{\fP}.
\]

\subsection{Hida families, Galois representations and specializations}

We briefly recall basic facts on Hida families of modular forms and their Galois representations. 

\subsubsection{Cohomology of modular curves}

Let $N\geq3$ be an integer and let $p\geq5$ be a prime such that $p\nmid N$. With notation as above, write $e^\prime_{\ord}\defeq\lim_{n\rightarrow \infty} (U_p^\prime)^{n!}$ for Hida's ordinary idempotent attached to $U_p^\prime$ and, following \cite[Proposition~7.2.1]{KLZ17}, define the module
\[
H^{1}_{\ord}(Np^\infty)\defeq e'_{\ord}\cdot\varprojlim_r H^1_{\text{\'et}}\Bigl(Y_1(Np^r)_{\overline{\Q}},\Zp(1)\Bigr)\simeq e'_{\ord}\cdot H^1_{\text{\'et}}\Bigl(Y_1(Np)_{\overline{\Q}},\Lambda\bigl(\sH_{\Zp}\langle t_{Np}\rangle\bigr)(1)\Bigr),
\]
where the sheaf $\Lambda\bigl(\sH_{\Zp}\langle t_{Np}\rangle\bigr)$ is defined in \cite[p.~33]{KLZ17}. It is proved in \cite[Proposition~7.2.1, (1)]{KLZ17} that $H^{1}_{\ord}(Np^\infty)$ is finitely generated and projective over $\Zp\llbracket \Zp^\times \rrbracket$. Moreover, $u\in \Zp^\times$ acts on $H^1_{\text{\'et}}\Bigl( Y_1(Np^r)_{\overline{\Q}},\Zp(1)\Bigr)$ as the diamond operator $\langle u^{-1}\rangle_{p^r}$. In view of this fact, we henceforth write $\Lambda_D = \Zp\llbracket \Zp^\times \rrbracket$, where the subscript $D$ stands for ``diamond operator''.

Let $I_{k,r}$ be the ideal of $\Lambda_D$ generated by ${[1+p^r]}-(1+p^r)^k$ for some $k\geq 0$ and $r\geq 1$, as in \cite[Proposition~7.2.1, (4)]{KLZ17}. The change of coefficient sheaves is given by the $k$-th moment map, which is the map 
\[
\mom^k: \Lambda\bigl(\sH_{\Zp}\langle t_{Np}\rangle\bigr)(1) \longrightarrow  \TSym^k(\sH_{\Zp})(1)
\]
defined in  \cite[Notation~4.4.3]{KLZ17}. This map induces an isomorphism of $\Zp$-modules
\[
H^1_{\ord}(Np^\infty)\big/I_{k,r} \overset\simeq\longrightarrow e'_{\ord} H^1_{\text{\'et}}\Bigl(Y_1(Np^r)_{\overline{\Q}}, \TSym^k(\sH_{\Zp})(1)\Bigr)
\]
(see \cite[Proposition~7.2.1, (4)]{KLZ17}).

\subsubsection{Hecke algebra and Hida families}

In this article, we denote by $\T_{Np^\infty}$ the \emph{Hecke algebra} generated over $\Z$ by the Hecke operators $T'_n$ for all $n\geq 1$. Since $\T_{Np^\infty}$ is a finite, projective $\Lambda_D$-algebra, it is isomorphic to the direct product of its localizations at its finitely many maximal ideals. Following the terminology of \cite{KLZ17}, these maximal ideals are referred to as \emph{Hida families}. Notice that there is a natural action of Hecke operators that turns $H^1_{\ord}(Np^\infty)$ into a $\T_{Np^\infty}$-module.

\subsubsection{$\Lambda$-adic Galois representations} \label{sec:lambdadic}

Let $\mbf$ be a Hida family of tame level $N_\mbf$ coprime to $p$ and dividing $N$. It is a foundational result of Hida (\cite{hida86b}) that associated with $\mbf$ there is a ``$\Lambda$-adic'' Galois representation; this representation, denoted by $M(\mbf)^*$ or $\rho_\mbf$, can be realized as a quotient of the \'etale cohomology group
\[
e'_{\ord} H^1_{\text{\'et}}\Bigl(Y_1(Np)_{\overline{\Q}},\Lambda\bigl(\sH_{\Zp}\langle t_{Np}\rangle\bigr)(1)\Bigr).
\]
More precisely, following \cite[Definition 7.2.5]{KLZ17}, we set
\[
M(\mbf)^* \defeq {H^{1}_{\ord}(Np^\infty)}_{\mbf},
\]
where the right-hand side denotes the submodule of $H^{1}_{\ord}(Np^\infty)$ on which the Hecke algebra $\T_{Np^{\infty}}$ acts via its localization at the maximal ideal $\mathfrak{m}_f$ associated with the residual Galois representation of $f$ (equivalently, the direct summand of $H_{\ord}^{1}(Np^{\infty})$ obtained by localizing at the maximal ideal $\mathfrak{m}_f$). Write $\Lambda_{\mbf}$ for the corresponding localization of $\T_{Np^\infty}$.

Henceforth, as in \cite[Definition 7.2.6]{KLZ17}, we make the following assumptions on $\mbf$:
\begin{enumerate}
\item $\mbf$ is \emph{non-Eisenstein} modulo $p$, \emph{i.e.}, the residual Galois representation $\rho_{\mbf}$ attached to $\mbf$ is irreducible (\emph{cf.} \cite[Definition~7.2.6]{KLZ17});
\item $\mbf$ is \emph{$p$-distinguished}, \emph{i.e.}, the semisimplification of $\overline{\rho_{\mbf}}|_{G_{\Q_p}}$ is the direct sum of two distinct characters.
\end{enumerate}
It is a consequence of results of Wiles that, under these assumptions, the ring $\Lambda_{\mbf}$ is Gorenstein and $M(\mbf)^*$ is free over $\Lambda_{\mbf}$ (see, \emph{e.g.}, \cite[Proposition~3.3.1]{EPW}, \cite[Theorem~4.3.4]{LLZ}, \cite[\S1]{ohta}). The family $\mbf$ being non-Eisenstein implies that $\mbf$ is cuspidal, so all the arithmetic specializations of $\mbf$ will be cuspidal.

\subsubsection{Arithmetic primes}

An \emph{arithmetic prime} of $\Lambda_{\mbf}$ is a prime ideal of height 1 over $I_{k,r}$;
it corresponds to an eigenform $f$ of weight $k+2\geq 2$ and level $Np^r$. Notation being as in \S\ref{section: Galois rep}, this comes together with a choice of a prime $\fP$ of $L$ above $p$ at which $f$ is ordinary. 

\subsubsection{Specializations}

The newform $f$ is a \emph{specialization} of the family $\mbf$: there is a specialization isomorphism
\[
\Sp_f: M(\mbf)^* \otimes_{\Lambda_{\mbf}} \cO_{\fP} \xrightarrow{\simeq} M_{\cO_{\fP}}(f)^*
\]
that fits into a commutative square 
\[
\xymatrix@C=35pt@R=35pt{H^1_{\ord}(Np^\infty)\ar[r]\ar[d]^-{\mom_k} & M(\mbf)^* \ar[d]^-{\Sp_f} \\
H^1_{\text{\'et}}\Bigl(Y_1(Np^r)_{\overline{\Q}},\TSym^k\sH_{\Zp}(1)\Bigr) \ar[r]
& M_{\cO_{L,\fP}}^*}
\]
in which $\TSym^k \sH_{\Z_p}$ is defined like $\TSym^k \sH_{\Q_p}$ from \S\ref{section: coeff sheaves} (\emph{cf.} \cite[\S7.3]{KLZ17}).

\subsection{Beilinson--Flach classes in Hida families}

Following \cite[\S1.4]{KLZ17}, we briefly review the construction of Beilinson--Flach classes.

\subsubsection{Beilinson--Flach classes: outline of the construction} \label{section: outline Section 1.4}

Let  $c>1$ be an integer such that $(c,6pN)=1$. Set $Y\defeq Y(M,N)[1/p]$ and write $Y^2$ for the self-product of $Y$ over $\mathbb{Z}[1/N, \Bmu_M]$.
For $M\,|\,N$, one can construct from the Eisenstein--Iwasawa class the Rankin--Iwasawa class
\[
\leftindex_c{\mathcal{RI}}^{[j]}_{M,n,a}\in H^3_{\text{\'et}}\Bigl(Y(M,N)^2_{\Z[1/MNp]}, \Lambda\bigl(\sH_{\Z_p}\langle t_N\rangle\bigr)^{[j,j]}(2-j)\Bigr),
\]
where 
\[
\Lambda\bigl(\sH_{\Z_p}\langle t_N\rangle\bigr)^{[j,j]}\defeq\Lambda\bigl(\sH_{\Z_p}\langle t_N\rangle\bigr)^{[j]} \boxtimes \Lambda\bigl(\sH_{\Z_p}\langle t_N\rangle\bigr)^{[j]}
=\Bigl(\Lambda(\sH_{\Z_p}\langle t_N\rangle) \otimes \TSym^j \sH_{\Zp}\Bigr)^{\boxtimes\,2}
\]
is the coefficient sheaf for $Y^2$ defined in \cite[p.~44]{KLZ17}.

\begin{prop}[Kings--Loeffler--Zerbes]\label{prop: alternating}
There is an equality
\[
\rho^*\Bigl(\leftindex_c{\mathcal{RI}}^{[j]}_{M,n,a}\Bigr)=(-1)^j\leftindex_c{\mathcal{RI}}^{[j]}_{M,n,-a},
\]
where $\rho$ is the involution of $Y^2$ that interchanges the two components.
\end{prop}

\begin{proof}
See \cite[Proposition~5.2.3, (i)]{KLZ17}.
\end{proof}

Moreover, there is a natural map
\[
Y(m,mN)^2 \longrightarrow Y_1(N)^2 \times \Spec\biggl( \Z\biggl[\Bmu_m,\frac{1}{mN}\biggr]\biggr).
\]
Roughly speaking, as explained in \cite[p.~10 and Theorem~6.3.4]{KLZ17}, the  Rankin--Iwasawa class gives, when projected to the ordinary part under the map above, the \emph{Beilinson--Flach class}
\[
\leftindex_c{\mathcal{BF}}_{m,N,a}\in(e'_{\ord},e'_{\ord})\cdot H^3_{\text{\'et}}\Bigl(Y_1(N)^2_{\Z[\Bmu_m, a/mNp]},\Lambda\bigl(\sH_{\Zp}\langle t_N\rangle\bigr)^{\boxtimes\,2} \otimes \Lambda_{\Gamma}(-\mbj)\Bigr).
\]
Here $\Lambda_{\Gamma}(-\mbj)$ is a pro-\'etale sheaf defined in \cite[Notation 6.3.3]{KLZ17}.
Its stalk at a geometric point is isomorphic to the Iwasawa algebra $\Lambda(\Gamma)$ of $\Gamma\defeq\Gal\bigl(\Q(\Bmu_{p^\infty})/\Q\bigr)\simeq\Zp^\times$, with $\Gamma$ acting via the inverse of the canonical character 
\[
\mbj:\Gamma\longrightarrow\Lambda(\Gamma)^\times
\]
that views the elements of $\Gamma$ as units of $\Lambda(\Gamma)$.

\subsubsection{The classes $\leftindex_c{\mathcal{BF}}_m^{\mbf,\mbf}$} \label{classes-subsubsec}

We specialize \cite[Definition 8.1.1]{KLZ17} to the $\mbf=\mbg$ case. Combining the Hochschild--Serre spectral sequence and the K\"unneth isomorphism yields a map
\begin{equation*}
\begin{split}
H^3_{\text{\'et}}\Bigl(& Y_1(N)^2_{\Z[\Bmu_m,a/mNp]},\Lambda\bigl(\sH_{\Zp}\langle t_N\rangle\bigr)^{\boxtimes\,2} \otimes \Lambda_{\Gamma}(-\mbj)\Bigr)\\
&\longrightarrow H^1\biggl(\Z\biggl[ \Bmu_m, \frac{1}{mNp}\biggr], H^1_{\text{\'et}}\Bigl(Y_1(N)^2_{\Z[\Bmu_m, a/mNp]}, \Lambda\bigl(\sH_{\Zp}\langle t_N\rangle\bigr)\Bigr)^{\!{\otimes\,2}} \otimes\Lambda_{\Gamma}(-\mbj)\biggr).
\end{split}
\end{equation*}
After projecting to the ordinary part, we obtain a class
\[
_c\mathcal{BF}_m^{\mbf,\mbf} \in H^1_{\text{\'et}}\Bigl(\Z[\Bmu_m, 1/mNp], M(\mbf)^* \otimes M(\mbf)^*\otimes\Lambda(\Gamma)(-\mbj)\Bigr).
\]
Let $\mbf$ be a Hida family whose tame level $N_\mbf$ divides $N$ and is not divisible by $p$. Define
\[
M(\mbf\otimes \mbf)^* \defeq M(\mbf)^* \hat{\otimes}_{\Zp} M(\mbf)^*,
\]
which is a $\Lambda_{\mbf}\,\hat{\otimes}\,\Lambda_{\mbf}$-module that is finite and projective over $\Lambda_D\,\hat{\otimes}\,\Lambda_D$. Then, via the pushforward along the degeneracy map
\[
(\pr_1 \times \pr_1)_*: Y_1(Np)^2 \longrightarrow Y_1(N_\mbf p) \times Y_1(N_\mbf p)
\]
and the K{\"u}nneth formula, we obtain a Galois-equivariant projection map 
\[
\pr_{\mbf,\mbf}: H^2_{\text{\'et}}\Bigl( Y_1(Np)^2_{\overline{\Q}}, \Lambda\bigl(\sH_{\Zp}\langle t_{Np}\rangle\bigr)^{\boxtimes \,2}(2)\Bigr) \longrightarrow M(\mbf \otimes \mbf)^*\simeq M(\mbf)^*\hat{\otimes} M(\mbf)^*
\]
(\cite[p.~72]{KLZ17}).
Thus, $M(\mbf\otimes \mbf)^*$ is a direct summand of $H^2_{\text{\'et}}\Bigl( Y_1(Np)^2_{\overline{\Q}}, \Lambda\bigl( \sH_{\Zp}\langle t_{Np}\rangle\bigr)^{\boxtimes\, 2}(2)\Bigr)$.

For all specializations $f,f'$ of $\mbf$, there is a commutative square 
\[
\xymatrix@C=40pt@R=40pt{
H^2_{\text{\'et}}\Bigl(Y_1(Np)^2_{\overline{\Q}},\Lambda\bigl( \sH_{\Zp}\langle t_{Np}\rangle\bigr)^{\boxtimes\,2}(2)\Bigr) \ar[r]^-{\pr_{\mbf,\mbf}} \ar[d]^-{\mom^k\boxtimes \mom^{k'}} 
& M(\mbf\otimes \mbf)^* \ar[d]^-{\Sp_f\otimes \Sp_{f'}} \\
H^2_{\text{\'et}}\Bigl(Y_1(Np^r)^2_{\overline{\Q}}, \TSym^{{[k,k']}}(\sH_{\Zp})(2)\Bigr) \ar[r]^-{\pr_{f,f'}} 
& M_{L,\fP}(f\otimes f')^*
}
\]
(\cite[\S8.1]{KLZ17}). Here the integer $r\geq1$ is large enough so that $f,f'$ have level $Np^r$ and 
\[
\TSym^{{[k,k']}}(\sH_{\Zp})\defeq\TSym^{{k}}(\sH_{\Zp}) \boxtimes \TSym^{{k'}}(\sH_{\Zp})
\]
is a sheaf on $Y_1(Np^r)^2$ (see, \emph{e.g.}, \cite[p.~4, \S3.2 and \S8.1]{KLZ17}).
For the goals of this paper, we are interested in the case where $f=f^\prime$.

Now let $m\geq 1$ and $c>1$ be integers such that $p\nmid m$ and $(c,6mpN_\mbf)=1$. Define
\[
_c\mathcal{BF}_m^{\mbf,\mbf} \in H^1\biggl(\Z\biggl[\frac{1}{mpN_\mbf^2},\Bmu_m\biggr], M(\mbf\otimes\mbf)^*\otimes\Lambda(\Gamma)(-\mbj)\!\biggr)
\]
to be the image of the class $_c\mathcal{BF}_{m,pN,1}$ under the map $\pr_{\mbf,\mbf}$.
We write $_c\mathcal{BF}^{\mbf,\mbf}$ for $_c\mathcal{BF}^{\mbf,\mbf}_1$.

It can be checked that the classes $_c\mathcal{BF}_m^{\mbf,\mbf}$ are independent of the choice of $N$ (\cite[Remark 8.1.2]{KLZ17}).
Now define $Q_{\ell}\in \Lambda_{\mbf}\hat{\otimes}\Lambda_{\mbf}[X,X^{-1}]$ by setting
\[ \begin{split}
    Q_\ell(X,X^{-1})\defeq&-X^{-1} + a_{\ell}(\mbf)^2+ \Bigl( (\ell+1)\ell^{\mbk + \mbk'} \varepsilon_{\mbf}(\ell)^2 - \bigl(\ell^{\mbk}+\ell^{\mbk'}\bigr)\varepsilon_{\mbf}(\ell)a_{\ell}(\mbf)^2 \Bigr)X\\&+ \ell^{\mbk + \mbk'} a_{\ell}(\mbf)^2 \varepsilon_{\ell}(\ell)^2X^2 - \ell^{1 + 2\mbk + 2\mbk'}\varepsilon_{\mbf}(\ell)^4X^3,
    \end{split}
\]
where $\varepsilon_\mbf:(\Z/N_\mbf\Z)^\times \rightarrow \Lambda_{\mbf}^\times$ is the prime-to-$p$ part of the Nebentypus of $\mbf$.

\begin{prop}[Kings--Loeffler--Zerbes] \label{originalBF}
Assume $p\nmid N_\mbf$. Let $\ell$ be a prime number such that $\ell\nmid cmpN_\mbf$ and let $\sigma_\ell\in\Gal(\Q(\Bmu_m)/\Q)$ denote the arithmetic Frobenius. Then
\[
\Norm_{m}^{\ell m}\Bigl(\leftindex_c{\mathcal{BF}}_{\ell m}^{\mbf,\mbf}\Bigr) =Q_\ell\Bigl(\ell^{-\mbj}\sigma_{\ell}^{-1}\Bigr)\cdot\leftindex_c{\mathcal{BF}}_m^{\mbf,\mbf}.
\]
\end{prop} 

\begin{proof} This is an incarnation of \cite[Theorem~8.1.3, (1)]{KLZ17} in our setting.
\end{proof}

\section{An Euler system for \texorpdfstring{$\Sym^2 M(\mbf)^*$}{}} \label{S: results 1}

Here, we introduce an Euler system for the symmetric square representation $\Sym^2 M(\mbf)^*$.

\subsection{Filtrations on Galois representations} \label{filt-subsec}

Consider the free $\Lambda_{\mbf}\,\hat{\otimes}\,\Lambda_{\mbf}$-module $M(\mbf\otimes \mbf)^*$ and the Galois representation $M(\mbf \otimes \mbf)^* [1/p]$ over the quotient field of $\Lambda_{\mbf}\,\hat{\otimes}\,\Lambda_{\mbf}$.
It  decomposes as follows into its symmetric square and alternating (or anti-symmetric) parts:
\[
M(\mbf \otimes \mbf)^*[1/p] = \Sym^2 M(\mbf)^* [1/p] \oplus \Exterior^2 M(\mbf)^*[1/p].
\]
We are assuming that the family $\mbf$ is non-Eisenstein, the residual Galois representation attached to $\mbf$ and the symmetric square representation associated to $\mbf$ are both irreducible (as this is true for the specialization at an arithmetic prime; see \cite[Note: 3.2.2]{LZ19} and \cite[p.~30, $(H.1^\sharp)$]{LZ19}). In this setting, $\Sym^2 M(\mbf)^*$ is the unique lattice inside $\Sym^2 M(\mbf)^* [1/p]$.
We have the following filtrations:
\begin{align*}
\sF^{++}M(\mbf\otimes\mbf)^* &\defeq \sF^+M(\mbf)^*\hat{\otimes} \sF^+M(\mbf)^*,\quad
\sF^{+-}M(\mbf\otimes\mbf)^* \defeq \sF^+M(\mbf)^* \hat{\otimes} \sF^-M(\mbf)^*,\\
\sF^{-+}M(\mbf\otimes\mbf)^* &\defeq \sF^-M(\mbf)^* \hat{\otimes} \sF^+M(\mbf)^*,\quad
\sF^{--}M(\mbf\otimes\mbf)^* \defeq \sF^-M(\mbf)^* \hat{\otimes} \sF^-M(\mbf)^*.
\end{align*}
These filtrations arise from those in \cite[Theorem~7.2.3]{KLZ17}; here we do not discuss them any further and simply refer to \cite{KLZ17} for their properties.

From here on, let $\psi$ be a Dirichlet character of conductor $N_\psi$ coprime to $p$. Let us define $\textbf{V}\defeq\Sym^2 M(\mbf)^*[1/p](1)(\psi)$ and consider the lattice
\[
\mbT\defeq\Sym^2 M(\mbf)^*(1)(\psi)\subset\textbf{V}.
\]
As before (\emph{cf.} \S\ref{section: outline Section 1.4}), let $c>1$ be an auxiliary integer with $(c,N_{\psi})=1$: it will essentially serve as a coprimality and normalization parameter. Moreover, set $\nu\defeq\psi(-1)$. The filtration on  $M(\mbf\otimes\mbf)^*$ induces a three-step filtration
\begin{align*}
\mbT=\sF^0 \mbT \supset \sF^1 \mbT \supset  \sF^2 \mbT \supset \sF^3 \mbT=0.
\end{align*}
One can write down $\sF^1$ and $\sF^2$ explicitly in terms of $\sF^\pm$ (see \cite[p.~31, equation (4.2)]{BO19}). We can also define the graded pieces $\Gr^i(\mathbf{T})\defeq\sF^i(\mathbf{T})/ \sF^{i+1}\mathbf{T}$. In particular, we have
\[
\Gr^0(\mbT)=\mbT/\sF^{1}(\mbT), 
\]
which projects isomorphically onto the quotient $ \sF^{--}M(\mbf \otimes\mbf)^*(1)$: this is analogous to what happens when one specializes a Hida family at an arithmetic prime (see, \emph{e.g.}, the proof of \cite[Proposition~4.2.1]{LZ19}).
Let 
\begin{equation} \label{Iw-eq}
H^1_{\Iw}\bigl(\Q(\Bmu_{mp^\infty}),\mbT\bigr)\defeq\varprojlim_r H^1\bigl(\Q(\Bmu_{mp^r}),\mbT\bigr) 
\end{equation} 
be the first Iwasawa cohomology group with values in $\mbT$, where the inverse limit is taken with respect to corestriction maps. Definition \eqref{Iw-eq} applies with $\Q(\Bmu_{mp^\infty})$ replaced by other fields or by products of finitely many fields, as will be clear later.

\subsection{An Euler system for \texorpdfstring{$\Sym^2 M(\mbf)^*$}{}}

Recall from \S\ref{filt-subsec} that, by definition, $\nu=\psi(-1)$. Set
\[
\sR\defeq\bigl\{\text{square-free integers $m\geq1$ such that $(m,6pcN_\mbf N_\psi)=1$}\bigr\}.
\]
Throughout this paper, write $\Q_{\cyc}$ for the unique $\Zp$-extension of $\Q$, so that $\Q_{\cyc}\subset\Q(\Bmu_{p^\infty})$. 

The following result, which corresponds to Theorem~\ref{A-intro} in the introduction, generalizes \cite[Theorem~4.1.6]{LZ19}. 

\begin{Th} \label{LZ 4.1.6}
There exists a family of cohomology classes ${\{c_m\}}_{m\in\sR}$ with
\[
c_m \in H^1_{\Iw}\bigl(\Q(\Bmu_{mp^\infty}),\mbT\bigr)^{\nu}
\]
such that $c_1 \neq 0$. Furthermore, if $\ell$ is a prime such that $m\in\sR$ and $\ell m\in\sR$, then
\[
\cores_{m}^{\ell m}(c_{\ell m}) = P_{\ell}\bigl( \sigma_{\ell}^{-1}\bigr)\cdot c_m,
\]
where $P_{\ell}$ is the Euler factor for $M(\mbf\otimes \mbf)(1)(\psi)$ given by
\[
P_{\ell}(\mbf \otimes\mbf \otimes \psi, X)\defeq\bigl(1-\ell^{\mbk-1} \psi\varepsilon_{\mbf}(\ell)X \bigr)\cdot P_{\ell}(\Sym^2\mbf\otimes\psi,X).
\]
In particular, the collection ${\{c_m\}}_{m\in\sR}$ is an Euler system.
\end{Th}

\begin{proof} The proof will be divided into several steps.
\begin{enumerate}[\textup{(}i\textup{)}]
\item \texttt{Reduction from $Q_\ell(X)$ to $P_\ell(X)$.} \newline 
We start with the Beilinson--Flach classes $_c \mathcal{BF}_{\ell m}^{\mbf,\mbf}$ satisfying the norm relations mentioned in Proposition~\ref{originalBF}.
Writing $P_{\ell}(X)$ for the Euler factor of $M(\mbf \otimes \mbf)(1)$ at $\ell$, we know by \cite[Remark~8.1.4]{KLZ17} that
\[
Q_{\ell}(X) = -X^{-1}P_{\ell}(X)\pmod{\ell-1}.
\]
In view of the discussion in \cite[\S7.3]{LLZ14} and of \cite[Lemma~9.6.1]{Rub_ES}, we can modify the classes $_c\mathcal{BF}_m^{\mbf,\mbf} $ by appropriate elements of $\Z_p[\Gal(\Q_{mp^\infty}/\Q)]^\times$ to get  classes
\[
c_m^{\star}\in H^1\Biggl(\Z\biggl[\frac{1}{mpN_\mbf^2},\Bmu_m\biggr], M(\mbf\otimes\mbf)^*\otimes\Lambda(\Gamma)(-\mbj)\!\Biggr)
\]
for all $m\geq 1$ coprime to $pcN_\mbf$, with $c_1^\star =$ $_c\mathcal{BF}_1^{\mbf,\mbf}$.
Furthermore, the compatibility relations
\[
\Norm_{m}^{\ell m}(c_{\ell m}^\star) = \begin{cases}
P_{\ell}(\ell^{-1}\sigma_{\ell}^{-1})\cdot c_m^\star & \text{if $\ell\,|\,pm$},\\[2mm]
c_m^\star & \text{otherwise}
\end{cases}
\]
hold, where $P_{\ell}$ is the Euler factor of $M(\mbf\otimes \mbf)(1)$ at $\ell$; this equality can be checked as in the proof of \cite[Theorem~11.4.1]{KLZ17}.
\item \texttt{Changing the normalization.}\newline
In practice, we can switch between the two normalizations 
$P_{\ell}(\ell^{-1}\sigma_{\ell}^{-1})$ and $P_{\ell}(\sigma_{\ell}^{-1})$ easily by
\cite[Remark 11.4.2]{KLZ17}. Therefore, we obtain classes for which the norm compatibility relation for $\ell \nmid pm$ is given by $P_{\ell}(\sigma_{\ell}^{-1})$.
\item \texttt{Construction of classes in the symmetric square.}\newline
Let $\rho$ be the automorphism of the variety $Y_1(Np)^2$ that interchanges the two factors. Via the K\"unneth isomorphism, there is an isomorphism 
\begin{equation}
\label{Kunneth iso}
H^2_{\text{\'et}}\Bigl(Y_1(Np)^2_{\overline{\Q}},\Lambda\bigl(\sH_{\Zp}\langle t_{Np}\rangle\bigr)^{\boxtimes\,2}(2)\!\Bigr)\simeq H^1_{\text{\'et}}\Bigl( Y_1(Np)_{\overline{\Q}},\Lambda\bigl(\sH_{{\Zp}}\langle t_{Np}\rangle\bigr)(1)\!\Bigr)^{\boxtimes\,2}
\end{equation}
(\cite[p. 16]{LZ19}). Then $\rho^*$ induces $-s$ on $M(\mbf\otimes \mbf)^*$.
But $M(\mbf\otimes \mbf)^*$ is preserved by the symmetry involution $s$ which interchanges the two factors on the right hand side of \eqref{Kunneth iso}.
It follows from the discussion in \S\ref{section: outline Section 1.4} and Proposition~\ref{prop: alternating} that $\rho^*\Bigl(\leftindex_c{\mathcal{BF}}^{\mbf,\mbf}_m\Bigr)=[\sigma]\cdot \leftindex_c{\mathcal{BF}}^{\mbf,\mbf}_m$, where $\sigma\in \Gal(\Q(\Bmu_{mp^\infty})/\Q)$ is the complex conjugation. Thus, there is an equality 
\[
s\Bigl(\leftindex_c{\mathcal{BF}}^{\mbf,\mbf}_m\Bigr)=-[\sigma]\cdot\leftindex_c{\mathcal{BF}}_m^{\mbf,\mbf}.
\]
Let $\chi$ be a continuous character of $\Gal(\Q(\Bmu_{mp^\infty})/\Q)\simeq(\Z/mp^\infty\Z)^\times$. If $\chi(-1)=-1$, then the image of $\leftindex_c{\mathcal{BF}}_m^{\mbf,\mbf}$ in $H^1\Bigl(\Q(\Bmu_m)^+,M(\mbf\otimes \mbf)^* (\chi)\!\Bigr)$ takes values in $\Sym^2 M(\mbf)^*(\chi)$. Now we construct the classes of the desired Euler system. For $m\in\sR$, let $c_m\in H^1_{\Iw}\bigl(\Q(\Bmu_{mp^\infty}),\mbT\bigr)^{\nu}$ be the image of $\leftindex_c{\mathcal{BF}}_{mN_\psi}^{\mbf,\mbf}$ under the natural twisting map
\[
H^1_{\Iw}\Bigl(\Q(\Bmu_{mN_\psi p^\infty}),\Sym^2 M(\mbf)^*\Bigr)^{(-1)}\longrightarrow H^1_{\Iw}\bigl(\Q(\Bmu_{mN_\psi p^\infty}), \mathbf{T}\bigr)^\nu\xrightarrow{\cores} H^1_{\Iw}\bigl(\Q(\Bmu_{mp^\infty}),\mbT\bigr)^{\nu}.
\]
\end{enumerate}
Combining the three steps above, we get the desired Euler system.

Finally, it is straightforward to see that our Euler system lifts to the symmetric square of a Hida family the Euler system that was attached by Loeffler--Zerbes to the symmetric square of a modular form (\cite[Theorem~4.1.6]{LZ19}): since the bottom class of the Euler system built in \cite{LZ19} is non-zero (\emph{cf.} \cite[Theorem~4.2.5]{LZ19}), the same is true for our Euler system as well. \end{proof}

We introduce notation for the $i$-th Greenberg Selmer groups. For $i\in\{0, \ldots, 3\}$, set
\begin{equation} \label{Gr-eq}
H^1_{\Gr,i}\bigl(\Q(\Bmu_{mp^\infty}),\mbT\bigr)\defeq\Bigl\{x\in H^1_{\Iw}\bigl(\Q(\Bmu_{mp^\infty}),\mbT\bigr)\mid\loc_p(x)\in H^1_{\Iw}\bigl(\Q(\Bmu_{mp^\infty})\otimes_\Q\Q_p,\sF^i\mbT\bigr)\Bigr\}.
\end{equation}
An analogous definition can be given with other fields in place of $\Q_{\cyc}$.

\begin{prop} \label{LZ 4.2.1}
The class $c_m$ belongs to $H^1_{\Gr,1}\bigl(\Q(\Bmu_{mp^\infty}),\mbT\bigr)^\nu$ for all $m\in\sR$.
\end{prop}

\begin{proof} As explained in the proof of \cite[Proposition~8.1.7]{KLZ17}, the image of the class $\leftindex_c{\mathcal{BF}}_m^{\mbf,\mbf}$ in $H^1\bigl(\Q_p, \sF^{--} M(\mbf\otimes \mbf)^*\otimes\Lambda(\Gamma)(-\mbj)\bigr)$ is trivial and $\Gr^0(\mbT) = \mbT/\sF^{1}(\mbT)$ projects isomorphically onto the quotient $\sF^{--}M(\mbf\otimes\mbf)^*(1)(\psi)$, whence the claim. \end{proof}

\subsection{Eliminating the factor \texorpdfstring{$c$}{}} \label{c-subsec}

Now we need to eliminate the auxiliary factor $c$. Write $\sP'$ for the set of prime numbers $\ell$ satisfying the following properties:
\begin{itemize}
\item $\ell \nmid pcN_\mbf N_{\psi}$;
\item $\ell\equiv 1 \pmod{p}$;
\item $\mbT/(\sigma_{\ell}-1)\mbT$ is a cyclic $\Zp$-module;
\item $\sigma_{\ell}-1$ is bijective on $\mbT'$, where $\mbT' \defeq \Exterior^2 M(\mbf)^*(1+\psi)$ is a lattice inside the representation $\mathbf{V}^{'}\defeq\Exterior^2 M(\mbf)^*[1/p](1+\psi)$.
\end{itemize}
Let $\sR'$ be the set of square-free products of primes in $\sP'$. Moreover, for all $m\in\sR'$ let $\Q(m)$ be the unique abelian $p$-extension of $\Q$ inside $\Q(\Bmu_m)$; the Selmer groups $H^1_{\Gr,i}\bigl(\Q(m)(\Bmu_{p^\infty}),\mbT\bigr)$ are defined as in \eqref{Gr-eq}.

\begin{Th} \label{ES we obtain}
There exists a collection of classes $\mathbf{c}'={(c'_m)}_{m\in\sR'}$ with
\[
c'_m\in H^1_{\Iw}\bigl(\Q(m)(\Bmu_{p^\infty}),\mbT\bigr)^\nu
\]
satisfying
\begin{enumerate}[\textup{(}i\textup{)}]
\item $c'_1=\leftindex_c{\mathcal{BF}}_{\psi}^{\mbf}\neq 0$;
\item $c'_m\in H^1_{\Gr,1}\bigl(\Q(m)(\Bmu_{p^\infty}),\mbT\bigr)^\nu$ for all $m\in\sR'$;
\item if $m\in \sR'$ and $\ell$ is a prime number such that $\ell m\in\sR'$, then
\[
\cores_m^{\ell m}(c'_{\ell m}) = P_{\ell}\bigl(\Sym^2 M(\mbf) \otimes \psi, \sigma_{\ell}^{-1}\bigr)\cdot c'_m.
\]
\end{enumerate}
\end{Th}

\begin{proof} Using Theorem~\ref{LZ 4.1.6}, proceed as in the proof of \cite[Theorem~5.3.3]{LZ19}. \end{proof}

\section{From Euler systems to Selmer bounds} \label{sec:Koly}

In this section, we apply the Kolyvagin system machinery to our Euler system to derive a divisibility of characteristic ideals for the relevant Greenberg Selmer groups.

\subsection{A big image result} 

First of all, we prove (under certain hypotheses) a big image result that is the analogue of \cite[Theorem~5.6]{BO19} and \cite[Proposition~5.2.1]{LZ19} in our setting.

Let $f$ be a non-CM newform of weight $k$ and level $N$, whose $q$-expansion will be denoted by $f(q)=\sum_{n\geq1}a_n(f)q^n$. Let $p$ be a good ordinary prime for $f$, \emph{i.e.}, $p\nmid N$ and $a_p(f)$ is a $p$-adic unit. Let $\mbf$ be the unique Hida family admitting the $p$-stabilization of $f$ as a weight $k$ arithmetic specialization. Suppose further that $f$ is $p$-distinguished. Let $\Lambda_\mbf$ be a power series ring (in one variable) and let $\overline{\rho}: G_\Q \rightarrow\GL_2(k)$ be the residual representation attached to $\mbf$.
Let 
\[
\bPsi:(\Z/pN\Z)^\times\longrightarrow\cO^\times
\]
be the central character defined by sending $\bPsi(\ell)$ to the eigenvalue of the diamond operator $\langle \ell \rangle$ acting on $\mbf$ (see, \emph{e.g.}, \cite[p.~40]{BO19}).
Here $\mathcal{O}$ is a suitable finite extension of $\Q_p$ containing the image of $\bPsi$. We can also enlarge $\cO$ so that it contains the image of the Dirichlet character $\psi$. As mentioned in \cite[equation~(5.1)]{BO19}, we have 
\[
\overline{\rho}|_{G_{\Q_p}} \simeq \begin{pmatrix}
\overline{\bPsi}\alpha^{-1} & *\\
0 & \alpha\\
\end{pmatrix},
\]
where $\alpha \in k$ is the reduction of the eigenvalue of the $U_p$-action on $\mbf$ modulo the maximal ideal $\fm$ of $\Lambda_\mbf$ (by an abuse of notation, we have denoted by $\alpha$ also the unramified character that assumes this value at the arithmetic Frobenius at $p$).

In the result that follows, notation from \S\ref{c-subsec} is in force.

\begin{Th}[Big image] \label{thm: big image}
Assume $p\geq 7$ and that
\begin{enumerate}[\textup{(}\text{BI}.1\textup{)}]
\item the central character $\bPsi$ is non-trivial; 
\item the residual representation $\overline{\rho}$ attached to $\mbf$ contains a conjugate of $\SL_2(\F_p)$; 
\item there exists $u\in (\Z/pN\Z )^\times$ with $\overline{\bPsi}(u)\neq 1$ such that $\psi(u)$ is a square in $\mathcal{O}^\times.$
\end{enumerate}
Then there exists $\tau\in G_{\Q}$ such that
\begin{itemize}
\item $\tau$ acts trivially on $\Bmu_{p^\infty}$;
\item $\mbT/(\tau-1)\mbT$ is a free $\Lambda_{\mbf}$-module of rank $1$;
\item $\tau-1$ acts invertibly on $\mbT'$.
\end{itemize}
\end{Th}

\begin{proof} Let $x$ be any element of $\Z_p\llbracket X\rrbracket^\times$ (viewed as an element of $\Z_p \llbracket \Gamma\rrbracket$ via the isomorphism $\Gamma \cong 1+\Z_p$), choose $u\in(\Z/pN\Z )^\times$ with $\overline{\bPsi}(u)\neq 1$ and let $\psi$ be a non-trivial Dirichlet character. The same argument as in the proof of \cite[equation (5.12), p.~45]{BO19} implies that the matrix
\[
D\defeq\begin{pmatrix}x^{-1}\psi^{-1/2}(u)&0\\0&x\psi^{1/2}(u)\bPsi(u)\end{pmatrix}
\]
lies in $\rho(G_{\Q(\Bmu_{p^\infty})})$. Let $\tau' \in \Q(\Bmu_{p^\infty})$ act via $D$. Upon multiplication by $x^2\psi$, the matrix $D\otimes D$ is the diagonal matrix $\Diag\bigl(1, x^2\psi(u)\bPsi(u),x^2\psi(u)\bPsi(u), x^4\psi^2(u)\bPsi^2(u)\bigr)$.
Therefore, $\tau'$ acts on the symmetric square piece $\mbT$ via the matrix $\Diag\bigl(1,x^2\psi(u)\bPsi(u), x^4\psi^2(u)\bPsi^2(u)\bigr)$, while the action of $\tau'$ on the anti-symmetric square piece $\mbT'$ is via $x^2\psi(u)\bPsi(u)$.
Choosing $x$ in a way that $x^2\psi(u)\bPsi(u) \not\equiv \pm 1\pmod{(p,X)}$, we obtain an element $\tau\in G_{\Q}$ with the desired properties. \end{proof}

\subsection{Assumptions} \label{ass-subsec}

Set $\mbA \defeq \mbT^\vee(1)$, where $(-)^\vee$ denotes Pontryagin dual. We equip $\mbA$ with a 3-step filtration, which we obtain by defining $\sF^{i}\mbA$ to be the orthogonal complement of $\sF^{3-i}\mbT$.

For simplicity, set $\Q_\infty\defeq\Q(\Bmu_{p^\infty})$; recall that $\Gamma\defeq\Gal(\Q_\infty/\Q)$ and set $\Gamma_1\defeq\Gal(\Q_{\cyc}/\Q)$, so that $\Gamma_1$ is a quotient of $\Gamma$. Actually, the canonical splitting $\Gamma=\Gamma_{\tor}\times\Gamma_1$ allows us to identify $\Gamma_1$ with a subgroup of $\Gamma$ as well. Consider the canonical character $\mbj:\Gamma\rightarrow\Lambda(\Gamma)^\times$ defined in \S\ref{section: outline Section 1.4}; by a slight abuse of notation, we use the same symbol for its restriction 
\[
\mbj:\Gamma_1\longrightarrow\Lambda(\Gamma_1)^\times
\]
to $\Gamma_1$. Fix a character $\eta$ of $\Gamma_{\tor}$ and set
\[
\T\defeq\mbT(\eta^{-1})\otimes\Lambda(\Gamma_1)(-\mbj).
\]
Note that $\T$ is equipped with the \emph{Greenberg Selmer structure}, denoted by $\sF_{\Gr,1}$, for which the local condition at $p$ is given by the cohomology of $\sF^{1}\T$; see \S\ref{S: results 1} for details.

Henceforth, in addition to the hypotheses in the statement of Theorem~\ref{thm: big image}, we assume the following:
\begin{itemize}
\item[(H.0)]  $H^0(\Q_p, \overline{\mbT})=0$;
\item[(H.2)]  $H^2(\Q_p, \overline{\mbT})=0$;
\item[(H.0a)]  $H^0(\Q_p, \overline{\mbT}/\Fil^{1}\overline{\mbT})=0$;
\item[(H.2a)]  $H^2(\Q_p, \overline{\mbT}/\Fil^{1}\overline{\mbT})=0$;
\item[(H.2b)]  $H^2(\Q_p, \overline{\mbT}/\Fil^{2}\overline{\mbT})=0$;
\item[(cond)] the conductor of $\bPsi$ is prime to $p$;
\item[(dist)] $\mbf$ is $p$-distinguished (\emph{i.e.}, the restriction to $G_{\Q_p}$ of the semisimplification of the residual representation of $\mbf$ is a direct sum of distinct characters);
\item[(Tam)] the $p$-part of the Tamagawa number (in the sense of \cite[\S7.6.10]{Nek-selmer}) at all primes different from $p$ for some classical specialization of $\mbf$ is $1$;
\item[(BG)] \cite[Assumption 2.7]{BG} holds.
\end{itemize}
In particular, conditions (Tam) and (BG) allow us to freely use results from \cite{BG}.

\subsection{A bound on Selmer groups}

From here on, write $e_\eta \Lambda_{\mbf}{\llbracket\Gamma\rrbracket}$ for $e_\eta\bigl(\Lambda_{\mbf}\,\hat\otimes\,\Lambda(\Gamma)\bigr)$.

\begin{Th} \label{to complete}
Let $\eta$ be a character of $\Gamma_{\tor}$ such that $\eta(-1)=\psi(-1)$. Then
\begin{enumerate}
\item The \emph{Greenberg Selmer group}
\[
e_{\eta}H^1_{\Gr,1}(\Q_{\infty},\mbT)\defeq e_\eta\biggl(\ker\Bigl(H^1_{\Iw}(\Q_{\infty}, \mbT)\longrightarrow H^1_{\Iw}\bigl(\Q_{p,\infty},\mbT/\sF^1\mbT\bigr)\!\Bigr)\!\biggr)
\]
is free of rank $1$ over $e_\eta\Lambda_{\mbf}{\llbracket\Gamma\rrbracket}$ and $e_\eta \cdot\leftindex_c{\mathcal{BF}}_{\psi}^{\mbf}$ is a non-torsion element of this module.
\item The \emph{Greenberg Selmer group}
\[
e_{\eta} H^1_{\Gr,2}(\Q_{\infty},\mbA)^\vee
\]
is ${e_\eta}\Lambda_{\mbf}\llbracket\Gamma\rrbracket$-torsion.
\item There is a divisibility of characteristic ideals
\[
\charac_{\Lambda_{\mbf}\llbracket\Gamma_1\rrbracket}\Bigl(e_{\eta}H^1_{\Gr,2}(\Q_{\infty},\mbA)^\vee\Bigr)\;\Big|\;\charac_{\Lambda_{\mbf}\llbracket\Gamma_1\rrbracket}\Biggl(\frac{e_{\eta}H^1_{\Gr,1}(\Q_{\infty},\mbT)}{\Lambda_{\mbf}\llbracket\Gamma_1\rrbracket\cdot\leftindex_c{\mathcal{BF}}_{\psi}^{\mbf}}\Biggr).
\]
\end{enumerate}
\end{Th}

\begin{proof} The proof of \cite[Theorem~2.26]{BO19} is entirely formal and is a generalization of that of \cite[Proposition~12.2.3]{KLZ17}, which works for ordinary $\Zp$-lattices: it takes the Euler systems of ordinary lattices over a multi-variable power series ring and produces a generalized Kolyvagin system; the resulting Kolyvagin classes satisfy Greenberg's local condition at $p$. 

Recall that in Theorem~\ref{ES we obtain} we proved the existence of the Euler system of $\Sym^2 M(\mbf)^*$. The arguments in the proof of \cite[Theorem~2.26]{BO19} combined with \cite[Appendix~A]{LZ19} yield a generalized Kolyvagin system, which we denote by $\overline{\mathbf{KS}}$ (in \cite{BO19}, whenever a local condition \emph{\`a la} Greenberg at $p$ is involved the authors denote this condition by a $+$ sign). Using notation of \cite{MR_KS}, let $\bkappa$ be an element of
\[
\overline{\mathbf{KS}}\bigl(\T, \sF_{\Gr, 1}, \mathcal{P}'\bigr),
\]
where $\mathcal{P}'$ is defined in \S\ref{c-subsec}. Observe that all the hypotheses listed before Theorem~\ref{to complete} are needed to apply the arguments in the proof of \cite[Theorem~2.26]{BO19}.

Write $\kappa_1$ for the $\eta$-isotypical projection of ${_c} \mathcal{BF}_\psi^{\mbf}$, which is non-zero by Theorem~\ref{ES we obtain}.
Recall from Proposition~\ref{LZ 4.2.1} and Theorem~\ref{ES we obtain} that the classes $c'_m$ satisfy the Greenberg local condition at $p$. On the other hand, the arguments in the proof of \cite[Theorem~5.3.4]{LZ19} show that the Kolyvagin classes must satisfy the same condition at $p$. Now the theorem follows as in the proof of \cite[Theorem~3.6]{BO19}, whose arguments are of a completely formal nature. \end{proof}

\section{An algebraic functional equation} \label{sec:AFE}

In this section, we prove an algebraic functional equation for the characteristic ideals of the Pontryagin duals of the Greenberg Selmer groups.

\subsection{Selmer groups and their Pontryagin duals} \label{selmer-subsec}

Recall the (localized) Hecke algebra $\Lambda_{\mbf}$ from \S \ref{sec:lambdadic}. Since $\mbf$ is assumed to be non-Eisenstein, the residual Galois representation attached to the symmetric square representation $\Sym^2\mbf$ associated with $\mbf$ is also irreducible, as this property holds for the specialization at an arithmetic prime (\cite[Note 3.2.2 and p.~208, $(H.1^\sharp)$]{LZ19}). 

We also assume that $\mbf$ is $p$-distinguished and write $\mbT_{\mbf}$ for the lattice $\Sym^2 M(\mbf)^*$, on which there is a three-step filtration. Moreover, there is a short exact sequence
\[
0\longrightarrow \sF^1(\mbT_{\mbf}) \longrightarrow \mbT_{\mbf} \longrightarrow \Gr^0(\mbT_{\mbf}) \longrightarrow 0.
\]
Define
\[
A_{\mbf}\defeq\mbT_{\mbf}\otimes_{\Lambda_{\mbf}}\Hom_{\cont}(\Lambda_{\mbf},\Q_p/\Zp).
\]
The filtration on $\mbT_{\mbf}$ induces a filtration on $A_{\mbf}$ and there is a short exact sequence
\[
0\longrightarrow \sF^1(A_{\mbf}) \longrightarrow A_{\mbf} \longrightarrow \Gr^0(A_{\mbf}) \longrightarrow 0.
\]
In the following lines, $L/\Q$ is an algebraic extension.
\begin{Defi}
The \emph{Selmer group of $A_{\mbf}$ over $L$} is  
\begin{align*}
S(A_{\mbf}/L) & \defeq \ker \biggl(H^1(\Q_S/L, A_{\mbf})\longrightarrow \bigoplus_{w\in S\smallsetminus S_p} H^1( I_w, A_{\mbf}) \oplus \bigoplus_{w|p} H^1\bigl(I_w,\Gr^0(A_{\mbf})\bigr)\!\biggr).    
\end{align*}
\end{Defi}
Let us consider the Pontryagin dual
\begin{equation} \label{X-eq}
X(\mbT_{\mbf}/L)\defeq\Hom_{\cont}\bigl(S(A_{\mbf}/L),\Q_p/\Zp\bigr)
\end{equation}
of $S(A_{\mbf}/L)$. The natural action of $\Gamma_1=\Gal(\Q_{\cyc}/\Q)$ on $X(\mbT_{\mbf}/\Q_{\cyc})$ equips $X(\mbT_{\mbf}/\Q_{\cyc})$ with a $\Lambda_{\mbf}\llbracket\Gamma_1\rrbracket$-module structure.
Furthermore, a standard application of the topological version of Nakayama's lemma (\cite[Lemma~5.2.18]{NSW}) shows that $X(\mbT_{\mbf}/\Q_{\cyc})$ is a finitely generated $\Lambda_{\mbf}\llbracket\Gamma_1\rrbracket$-module. We also introduce the dual module
\[
\mbT_{\mbf}^* \defeq \Hom_{\Lambda_{\mbf}}\bigl(\mbT_{\mbf},\Lambda_{\mbf}(1)\bigr)
\]
and define the Selmer group attached to $\mbT_{\mbf}^*$ in an analogous fashion.

For simplicity, from here on we assume that 
$\Lambda_\mbf$ is a power series ring.
Let us write $f_k$ for the specialization of $\mbf$ of weight $k$ and trivial character and let $T_{f_k}$ be the Tate module of the corresponding symmetric square representation.

\begin{prop} \label{prop: Jha-Pal Thm 4.1}
The kernel and the cokernel of the natural specialization map
\[
X(\mbT_{\mbf}/\Q_{\cyc})\big/\nu_k X(\mbT_{\mbf}/\Q_{\cyc}) \xrightarrow{s_k^\vee} X(T_{f_k}/\Q_{\cyc})
\]
are finitely generated for all arithmetic primes $\nu_k$. Moreover, they are finite for all but finitely many arithmetic primes and there is an equality of characteristic ideals
\[
\charac_{\mathcal{O}_{\Q(f_k)}\llbracket T\rrbracket}\bigl( X(\mbT_{\mbf}/\Q_{\cyc})/\nu_k X(\mbT_{\mbf}/\Q_{\cyc})\bigr) = \charac_{\mathcal{O}_{\Q(f_k)}\llbracket T\rrbracket}\bigl( X(T_{f_k}/\Q_{\cyc})\bigr)
\]
for all but finitely many $k$.
\end{prop}

\begin{proof}[Sketch of proof]
The proof is similar to that of \cite[Theorem~4.1]{JP14}, so we only offer a sketch of it. Since we are working over $\Q$, the dual Selmer groups $X(T_{f_k}/\Q_{\cyc})$ and $X(T_{f_k}^*/\Q_{\cyc})$ are torsion modules over the corresponding Iwasawa algebra, by a result of Kato (\cite{Kato}). Now consider the commutative diagram with exact rows
\[
\begin{footnotesize}
\xymatrix@C=20pt{
0\ar[r]& S(A_{\mbf}/\Q_{\cyc}) \ar[d]^-{s_k}\ar[r]& H^1(\Q_S/\Q_{\cyc},A_{\mbf}) \ar[d]^-{\eta_k}\ar[r]& \displaystyle{\bigoplus_{w\in S\setminus S_p} H^1(I_w,A_{\mbf}) \oplus \bigoplus_{w|p} H^1\Bigl(I_w,\Gr^0(A_{\mbf})\Bigr) \ar[d]^-{\gamma_k=\oplus \gamma_{k,v}}}\\
0\ar[r]& S(A_{\mbf}/\Q_{\cyc})[\nu_k] \ar[r]& H^1(\Q_S/\Q_{\cyc},A_{\mbf})[\nu_k] \ar[r]& \displaystyle{\bigoplus_{w\in S\setminus S_p} H^1(I_w,A_{\mbf})[\nu_k]\oplus \bigoplus_{w|p} H^1\bigl(I_w,\Gr^0(A_{\mbf})\Bigr)[\nu_k].}}
\end{footnotesize}
\]
We are interested in the kernel and the cokernel of $s_k$. Given a finitely generated $\Lambda_{\mbf}$-module $M$, let us set 
\[
M^{\dagger}\defeq\Hom_{\Lambda_{\mbf}}(M,\Lambda_{\mbf}). 
\]
With this notation in hand, we see that
\[
\ker(\eta_k)^\vee \simeq \Bigl(A_{\mbf}^{G_{\Q_{\cyc}}}/\nu_k\Bigr)^{\!\vee} \simeq \bigl(\mbT_{\mbf}^{\dagger}\bigr)_{ G_{\Q_{\cyc}}}[\nu_k].
\]
Since $\bigl(\mbT_{\mbf}^{\dagger}\bigr)_{ G_{\Q_{\cyc}}}$ is a finitely generated $\Lambda_{\mbf}$-module, for every $\nu_k$ the $\Zp$-module $\bigl(\mbT_{\mbf}^{\dagger}\bigr)_{ G_{\Q_{\cyc}}}[\nu_k]$ is finitely generated. If $\bigl(\mbT_{\mbf}^\dagger\bigr)_{G_{\Q_{\cyc}},\tor}$ is the torsion submodule of $\big(\mbT_{\mbf}^{\dagger}\big)_{ G_{\Q_{\cyc}}}$, then
\begin{equation} \label{iso-eq}
\ker(\eta_k)^\vee \simeq \big(\mbT_{\mbf}^{\dagger}\big)_{ G_{\Q_{\cyc}}}[\nu_k]=\bigl(\mbT_{\mbf}^\dagger\bigr)_{G_{\Q_{\cyc}},\tor}[\nu_k].
\end{equation}
Recall that we are assuming that $\Lambda_\mbf$ is a power series ring; apply the Weierstrass Preparation Theorem (\cite[Theorem~5.3.4]{NSW}) to the characteristic series of $\bigl(\mbT_{\mbf}^\dagger\bigr)_{G_{\Q_{\cyc}},\tor}$, which is viewed as a $\Lambda_\mbf$-module. If $S$ is the finite set of irreducible polynomials appearing in this characteristic series, then $\bigl(\mbT_{\mbf}^\dagger\bigr)_{G_{\Q_{\cyc}},\tor}[\nu_k]$ is finite for all $\nu_k \not\in S$; by \eqref{iso-eq}, this finiteness ensures that $\ker(\eta_k)^\vee$ is finite for all but finitely many $k$. Finally, the snake lemma implies that $\ker(s_k)$ is finite for all but finitely many $k$. Next, note that $\coker(\eta_k)$ is trivial. With arguments analogous to those above, one can show that $\ker(\gamma_{k,\nu})$ is finite for all $\nu\in S$, and the proposition follows. \end{proof}

We record a consequence on (Pontryagin duals of) Selmer groups.

\begin{prop}\label{torbig}
The  dual Selmer group $X(\mbT_{\mbf}/\Q_{\cyc})$ is a finitely generated torsion $\Lambda_\mbf\llbracket\Gamma_1\rrbracket$-module.
\end{prop}
\begin{proof}
Applying Proposition~\ref{prop: Jha-Pal Thm 4.1}, we may choose an arithmetic prime $\nu_k$ such that the kernel and the cokernel of the map $s_k^\vee$ are finite.
By a result of Kato \cite{Kato}, $X(T_{f_k}/\Q_{\cyc})$ is torsion over $\cO_{\Q(f_k)}\llbracket\Gamma_1\rrbracket$. The proposition follows by noting that the arithmetic prime $\nu_k$ is a height $1$ prime ideal of $\Lambda_\mbf\llbracket\Gamma_1 \rrbracket$ and hence $X(\mbT_{\mbf}/\Q_{\cyc})$ cannot contain a $\Lambda_\mbf\llbracket \Gamma_1\rrbracket$-free submodule. \end{proof}

\begin{Remark}\label{remark:dual}
Propositions \ref{prop: Jha-Pal Thm 4.1} and \ref{torbig} are also true for $X(\mbT_{\mbf}^*/\Q_{\cyc} )^\iota$ in place of $X(\mbT_{\mbf}/\Q_{\cyc})$: the same proofs apply.
\end{Remark}

The following proposition, which is essentially due to Ochiai (\emph{cf.} \cite[\S3]{Ochiai-fourier}), provides sufficient conditions for obtaining an algebraic functional equation over a two-variable Iwasawa algebra using a descent argument to a one-variable setting. 

\begin{prop}[Ochiai]
\label{prop: JP 4.7}
Let $\mathcal{O}$ be the ring of integers of a finite extension of $\Q_p$, let $\Lambda^2_{\mathcal{O}}\defeq\mathcal{O}\llbracket T_1, T_2\rrbracket$ be a two-variable power series $\mathcal O$-algebra and let $M,N$ be finitely generated torsion $\Lambda^2_{\mathcal{O}}$-modules. Let ${\{\ell_i\}}_i$ be a countably infinite set of height $1$ prime ideals of $\mathcal{O}\llbracket T_1 \rrbracket$ such that 
\begin{enumerate}[\textup{(}i\textup{)}]
    \item there is a finite extension $\mathcal{O}'$ of $\mathcal{O}$ such that $\mathcal{O}\llbracket T_1\rrbracket/\ell_{i} \subset \mathcal{O}'$ for each $i$;
    \item for each $i$, both $M/\ell_i$ and $N/\ell_i$ are torsion $\Lambda_{\mathcal{O}}^2/\ell_i$-modules;
    \item for each $i$ and $W\in\{M,N\}$, the image of $\charac_{\Lambda_{\mathcal{O}}^2}(W)$ in $\Lambda_{\mathcal{O}^2}/\ell_i$ is equal to $\charac_{\Lambda_{\mathcal{O}}^2/\ell_i}(W/\ell_i)$ as ideals of $\Lambda_{\mathcal{O}}^2/\ell_i$.
\end{enumerate}
Finally, suppose that $\charac_{\Lambda_{\mathcal{O}}^2/\ell_i}(M/\ell_i)=\charac_{\Lambda_{\mathcal{O}}^2/\ell_i}(N/\ell_i)$ for each $i$. Then there is an equality $\charac_{\Lambda_{\mathcal{O}}^2}(M)=\charac_{\Lambda_{\mathcal{O}}^2}(N)$ of characteristic ideals.
\end{prop}

\begin{proof}
This is \cite[Proposition~4.7]{JP14}.
\end{proof}

\subsection{An algebraic functional equation} \label{alg-subsec}

Recall from \S\ref{sec:lambdadic} that we are assuming throughout that $\mbf$ is non-Eisenstein and $p$-distinguished; moreover, recall that, in this section, $\Lambda_\mbf$ is taken to be a power series ring. Now we prove the algebraic functional equation we are interested in; see also \S\ref{alg-subsec2} for a different, somewhat more abstract perspective on algebraic functional equations.

\begin{Th} \label{thm: alg func eqn}
Assume that $X(\mbT_{\mbf}/\Q_{\cyc})$ and $X(\mbT_{\mbf}^*/\Q_{\cyc})^{\iota}$ have no non-zero pseudonull $\Lambda_{\mbf}\llbracket\Gamma_1\rrbracket$-submodules. Then there is an equality
\[
\charac_{\Lambda_{\mbf}\llbracket\Gamma_1\rrbracket}\bigl(X(\mbT_{\mbf}/\Q_{\cyc})\bigr)=\charac_{\Lambda_{\mbf}\llbracket\Gamma_1\rrbracket}\bigl(X(\mbT_{\mbf}^*/\Q_{\cyc})^{\iota}\bigr).
\]
\end{Th}
\begin{proof}
We apply Proposition~\ref{prop: JP 4.7} with $M = X(\mbT_{\mbf}/\Q_{\cyc})$ and $N = X(\mbT_{\mbf}^*/\Q_{\cyc})^{\iota}$. With notation as in the statement of Proposition~\ref{prop: JP 4.7}, our choice of ${\{\ell_i\}}_i$ requires fixing an arbitrary positive integer $r_0$. Namely, the (countably infinite) set of height 1 primes we consider here is the set of all arithmetic points $\nu_k$ such that the corresponding specializations $f_k$ have level $Np^r$ for all $r\leq r_0$. This means that there exists a finite extension $\cO'$ of $\Zp$ such that $\cO_{f_k}\subset\cO'$ for every arithmetic prime $\nu_k$. By Proposition~\ref{prop: Jha-Pal Thm 4.1}, the kernel and the cokernel of the map
\[
X(\mbT_{\mbf}/\Q_{\cyc})\big/\nu_k X(\mbT_{\mbf}/\Q_{\cyc})\xrightarrow{s_k^\vee} X(T_{f_k}/\Q_{\cyc})
\]
are finite for all but finitely many arithmetic primes. Moreover, by \cite{Kato}, we know that $X(T_{f_k}/\Q_{\cyc})$ is torsion over $\cO_{\Q(f_k)}\llbracket\Gamma_1\rrbracket$, so the quotient
\[
X(\mbT_{\mbf}/\Q_{\cyc})\big/\nu_k X(\mbT_{\mbf}/\Q_{\cyc})
\] 
is also torsion over $\cO_{\Q(f_k)}\llbracket\Gamma_1\rrbracket$. In view of Remark \ref{remark:dual}, similar arguments lead us to conclude that
\[
X(\mbT_{\mbf}^*/\Q_{\cyc})^{\iota}\big/\nu_k X(\mbT_{\mbf}^*/\Q_{\cyc})^{\iota}
\]
is torsion over $\cO_{\Q(f_k)}\llbracket\Gamma_1\rrbracket$ as well. Thus, condition (ii) of Proposition~\ref{prop: JP 4.7} is satisfied. Moreover, notice that condition (iii) in Proposition~\ref{prop: JP 4.7} is true because, by assumption, $X(\mbT_{\mbf}/\Q_{\cyc})$ and $X(\mbT_{\mbf}^*/\Q_{\cyc})^{\iota}$ have no non-zero pseudonull $\Lambda_{\mbf}\llbracket \Gamma\rrbracket$-submodules. Therefore, it suffices to show that 
\begin{align*}
 \charac_{\mathcal{O}_{\Q(f_k)}\llbracket T\rrbracket}\bigl( X(\mbT_{\mbf}/\Q_{\cyc})/\nu_k X(\mbT_{\mbf}/\Q_{\cyc})\bigr) 
 &=\charac_{\mathcal{O}_{\Q(f_k)}\llbracket T\rrbracket}\bigl( X(\mbT_{\mbf}^*/\Q_{\cyc})^\iota/\nu_k X(\mbT_{\mbf}^*/\Q_{\cyc})^\iota\bigr)\\
 &=\charac_{\mathcal{O}_{\Q(f_k)}\llbracket T\rrbracket}\bigl( X(\mbT_{\mbf}^*/\Q_{\cyc})/\nu_k X(\mbT_{\mbf}^*/\Q_{\cyc})\bigr)^\iota.
\end{align*}
Again, by Proposition~\ref{prop: Jha-Pal Thm 4.1} and Remark \ref{remark:dual} we are reduced to showing that 
\begin{equation} \label{char-final-eq}
\charac_{\mathcal{O}_{\Q(f_k)}\llbracket T\rrbracket}\bigl( X(T_{f_k}/\Q_{\cyc})\bigr)=\charac_{\mathcal{O}_{\Q(f_k)}\llbracket T\rrbracket}\bigl(X(T_{f_k}^*/\Q_{\cyc})\bigr)^\iota
\end{equation}
for all but finitely many $k$. To conclude, observe that equality \eqref{char-final-eq} is the algebraic functional equation for $X(T_{f_k}/\Q_{\cyc})$ that was proved (in greater generality) in \cite[Theorem~2]{Gre89}. 
\end{proof}

\begin{Remark}
We expect it should be possible to avoid the assumption that  $X(\mbT_{\mbf}/\Q_{\cyc})$ and $X(\mbT_{\mbf}^*/\Q_{\cyc})^{\iota}$ have no non-zero pseudonull $\Lambda_{\mbf}\llbracket\Gamma_1\rrbracket$-submodules provided we can show the analogue of \cite[Proposition~4.8]{JP14} in our setting.
\end{Remark}

\section{Towards the Iwasawa main conjecture} \label{div-sec}

In this section, we prove one divisibility in the Iwasawa main conjecture for the symmetric square of a Hida family. As will be apparent, our proof builds crucially on recent work by B\"uy\"ukboduk--Ganguly on functional equations of algebraic $p$-adic $L$-functions (\cite{BG}).

\subsection{\texorpdfstring{$p$}{}-adic \texorpdfstring{$L$}{}-functions} \label{S:pathologies}

We briefly introduce the $p$-adic $L$-functions that will appear in our divisibility result in the direction of the Iwasawa main conjecture for $\Sym^2\mbf$. 


\subsubsection{Perrin-Riou's big dual exponential map} \label{PR-subsubsec}

The \emph{Perrin-Riou big dual exponential map} (or \emph{big logarithm map}) is the map
\[
\mathcal{L}:H^1_{\Iw}\bigl(\Q_{p,\infty},\Gr^1\mbV\bigr)^{{\nu}}\longrightarrow \Lambda(\Gamma)^{\nu}\otimes\D_{\cris}(\Gr^1\mbV)
\]
obtained by applying the (formal) argument in \cite[Theorem~8.2.8]{KLZ17} (see also \cite[Theorem~5.15]{BO19}); here $\nu$ and $\Gr^1(-)$ are as in \S\ref{S: results 1}. Define the homomorphisms of $\Lambda_{\mbf}$-modules $\omega_{\mbf}$ and $\eta_{\mbf}$ as in the statement of \cite[Proposition~10.1.1]{KLZ17}, where we write $\eta_{\mbf}$ in place of $\eta_{\mba}$. Furthermore, set
\[
\xi_{\mbf}\defeq\frac{1}{2}(\eta_{\mbf}\otimes\omega_{\mbf}+\omega_{\mbf}\otimes \eta_{\mbf}),\quad\xi_{\mbf,\psi}\defeq G(\psi^{-1})\cdot \xi_{\mbf}\in \D_{\cris}\bigl(V^{*}(1)\bigr),
\]
where $G(\psi^{-1})$ is the classical Gauss sum associated with $\psi^{-1}$.

\subsubsection{Three-variable \texorpdfstring{$p$}{}-adic \texorpdfstring{$L$}{}-functions and factorization} \label{3-subsubsec}

Let $\mbf$ and $\mbg$ be two Hida families and, with notation as in \cite[\S3.4]{Das16}, let $\mathcal W_{\mbf}$ and $\mathcal W_{\mbg}$ be the associated rigid-analytic spaces that are finite covers of two suitable connected components of the $p$-adic weight space $\mathcal W$, whose set of $\C_p$-points is given by $\mathcal W(\C_p)\defeq\Hom_{\mathrm{cont}}(\Zp^\times,\C_p^\times)$. Let $\psi$ be a Dirichlet character. In \cite{Hida-fourier}, Hida constructed a three-variable $p$-adic $L$-function $L_p(\mbf,\mbg,\psi,\kappa,\lambda,\sigma)$ on $\mathcal W_{\mbf}\times\mathcal W_{\mbg}\times\mathcal W$; this function interpolates the algebraic parts of the critical values of certain Rankin $L$-series of the specializations of $\mbf$ and $\mbg$ at $\kappa$ and $\lambda$, respectively, twisted by $\psi$. From here on, we specialize to the case $\mbf=\mbg$ and write
\[ L_p(\mbf,\mbf,\psi,1+\mbj)\in I_{\mbf}\,\widehat{\otimes}\, \Lambda_{\mbf}\,\widehat{\otimes}\,\Lambda(\Gamma)\otimes_{\Zp}\Zp[\Bmu_N] \otimes \text{im}(\psi) \] 
for the three-variable Rankin--Selberg $p$-adic $L$-function mentioned in \cite[Theorem~7.7.2]{KLZ17}.

On the half of the weight space where $\psi(-1)=-\sigma(-1)$, Hida similarly defined a three-variable $p$-adic $L$-function attached to the symmetric square $\Sym^2\mbf$ of $\mbf$, which we denote by $L_p(\Sym^2\mbf,\psi,\kappa,\sigma)$. A crucial ingredient for our purposes is the relationship between these two analytic objects. By the work of Dasgupta (\cite{Das16}), the Rankin--Selberg $p$-adic $L$-function factors over this half of the weight space as
\begin{equation}
\label{L-fact-eq}
L_p(\mbf,\mbf,\psi,1+\mbj)=L_p(\Sym^2\mbf,\psi,\kappa,\sigma)\cdot L_p(\psi\epsilon_{\mbf},s),
\end{equation}
where $L_p(\psi\epsilon_{\mbf},s)$ is the corresponding three-variable Kubota--Leopoldt $p$-adic $L$-function attached to the Dirichlet character $\psi\epsilon_{\mbf}$; here, as in \S\ref{classes-subsubsec}, $\varepsilon_\mbf$ is the prime-to-$p$ part of the Nebentypus of $\mbf$. It is exactly factorization \eqref{L-fact-eq} that allows us to extract the symmetric square $p$-adic $L$-function from the regulator image of the Beilinson--Flach classes in \S\ref{div-sec}. 

\subsection{Selmer complexes and the algebraic functional equation} \label{alg-subsec2}

In \S\ref{alg-subsec}, we established the algebraic functional equation for the classical Greenberg dual Selmer groups via a descent argument. However, in order to deduce one divisibility in the Iwasawa main conjecture for the symmetric square of a Hida family we must bypass the parity obstruction preventing a direct comparison of the Kolyvagin bound from Theorem~\ref{to complete} with the $p$-adic $L$-function. Achieving this requires us to lift the functional equation to the level of derived categories. To do so, we employ the Iwasawa-theoretic Selmer complexes developed by B\"uy\"ukboduk--Ganguly in \cite{BG}. This machinery guarantees the necessary perfectness and amplitude conditions to establish an equality of determinants, which transports our bound across the weight space.

In the lines below, we use notation from \cite{BG} freely: the reader is advised to keep a copy of \cite{BG} close at hand. With this \emph{caveat} in mind, let $T\defeq\Sym^2 T_f^*(1+\psi)$ and define its cyclotomic deformation $\overline{T}\defeq T\otimes_{R_f}\!R^\iota$. Attached to the data $\bigl(\overline{T}, \mathcal{F}^i\overline{T}\bigr)$, we consider the Greenberg Selmer complex $\mathrm{\mathbf{R}} \Gamma\bigl(\overline{T},\mathcal{F}^i\overline{T}\bigr)$ for $i\in\{1,2\}$. First, we build a bridge between the cohomology of these Selmer complexes and the classical Greenberg Selmer groups $X(T_f/\Q_{cyc})$ studied in \S\ref{sec:AFE}.

\begin{Lemma} \label{surj-lemma}
For $i \in \{1, 2\}$, there exists a surjection
\begin{equation} \label{surj-eq}
R^1\Gamma\Bigl(\overline{T}^\vee(1),\overline{\mathcal{F}^i T^\vee}(1)\!\Bigr)\longepi H^1_{\Gr,i}\bigl(\Q(\Bmu_{p^\infty}),T^\vee(1)\bigr).
\end{equation}
Furthermore, the kernel of this map is cotorsion as an $R_f[\![\Gamma]\!]$-module.
\end{Lemma}

\begin{proof} This is \cite[Lemma~2.19, (ii)]{BG}. Namely, as shown in \cite[(2.9)]{BG}, the explicit exact sequence naturally derived from the dual Selmer complex yields an exact sequence
\begin{equation} \label{H-eq}
\begin{split}
H^0\Bigl(G_{\Q_p},\overline{T}^\vee(1)\big/\overline{\mathcal{F}^{3-i}T}^\vee(1)\!\Bigr) &\longrightarrow R^1\Gamma\bigl(\overline{T}^\vee(1),\mathcal{F}^i\overline{T}^\vee(1)\bigr)\\ &\longrightarrow H^1_{\Gr,i}\bigl(\Q(\Bmu_{p^\infty}),T^\vee(1)\bigr)\longrightarrow 0.
\end{split}
\end{equation}
Since the $H^0$ term in \eqref{H-eq} is cotorsion, the kernel of our desired surjection is cotorsion too. \end{proof}

Notice that, as a consequence of Lemma~\ref{surj-lemma}, the Pontryagin dual of the kernel of \eqref{surj-eq} has a trivial characteristic ideal in $R_f[\![\Gamma]\!]$.

For the next result, we impose the big image hypothesis $(\mathrm{\mathbf{BI}}_{\Sym})$ from 
\cite[\S2.5.7]{BG}.

\begin{prop} \label{big-prop}
Assume that $(\mathrm{\mathbf{BI}}_{\Sym})$ holds and let $\eta$ be a character of $\Delta_\Q$ with $\eta(-1) = \psi(-1)$. Then:
\begin{enumerate}
\item $e_\eta R^1\Gamma\bigl(\overline{T},\mathcal{F}^i\overline{T}\bigr)=0$ for $i\in\{1,2\}$;
\item $e_\eta R^2\Gamma\bigl(\overline{T},\mathcal{F}^2\overline{T}\bigr)$ is a torsion module over $R_f[\![\Gamma_1]\!]$.
\end{enumerate}
\end{prop}

\begin{proof} This is \cite[Theorem~2.20]{BG}. \end{proof}

Our assumptions (specifically, the big image hypothesis and condition (Tam), \emph{cf.} \S\ref{ass-subsec}) satisfy the requirements of \cite[Proposition~2.16 and Theorem~2.21]{BG}, so the relevant Selmer complexes are perfect and we obtain the following algebraic functional equation.

\begin{Th}[Algebraic functional equation] \label{func-thm}
Under the hypotheses of Proposition~\ref{big-prop}, there is an equality of characteristic ideals
\[
\charac_{R_f[\![\Gamma_1]\!]}\Bigl(R^2\Gamma\bigl(\overline{T},\mathcal{F}^2\overline{T}\bigr)\!\Bigr)=\charac_{R_f[\![\Gamma_1]\!]}\Bigl(R^2\Gamma\bigl(\overline{T}^*(1),\overline{\mathcal{F}^1 T^*}(1)\bigr)\!\Bigr)^\iota.
\]
Consequently, the characteristic ideals of the corresponding classical Greenberg dual Selmer groups satisfy the same functional equation.
\end{Th}

\begin{proof} To begin with, the main duality result for adjoint representations in Hida families, as proved in \cite[Theorem~2.21]{BG}, establishes an equality between the determinants of the respective Selmer complexes. By Proposition~\ref{big-prop}, the cohomology of these complexes is concentrated in degree $2$ (since $R^1\Gamma$ vanishes) and this degree-$2$ cohomology is a torsion module over $R_f[\![\Gamma]\!]$. Because of this amplitude restriction and perfectness, the determinant of each Selmer complex is canonically identified exactly with the characteristic ideal of its degree-$2$ cohomology. This yields the equality of the characteristic ideals for the complexes. Finally, Lemma~\ref{surj-lemma} ensures that the classical Selmer groups differ from the degree-$2$ cohomology of the complexes only by pseudo-null (cotorsion) modules. Therefore, their characteristic ideals identically coincide, transporting the functional equation to the classical setting. \end{proof}

The result above corresponds to Theorem~\ref{B-intro} in the introduction.

\subsection{Towards a divisibility in the Iwasawa main conjecture}

Before stating our result towards the Iwasawa main conjecture divisibility on the required half of the weight space, we fix the characters indexing the branches of the Iwasawa algebra. Let $\eta$ be a character of $\Gamma_{\tor}$ satisfying the parity condition $\eta(-1)=\psi(-1)$, which is needed to extract a non-trivial bound from the Euler system (as in Theorem~\ref{to complete}). Moreover, denote by $\sigma$ the dual character corresponding to the involution $\iota$ and by $e_\sigma$ the associated idempotent; consequently, $\sigma$ satisfies $\sigma(-1)=-\psi(-1)$, in accord with the analytic parity condition required for the domain of the three-variable symmetric square $p$-adic $L$-function $L_p(\Sym^2\mbf,\psi,\kappa,\sigma)$ introduced in \S\ref{3-subsubsec}. Finally, recall the Kubota--Leopoldt $p$-adic $L$-function $L_p(\psi\epsilon_{\mbf},s-k+1)$ from \S\ref{3-subsubsec} and the dual Selmer group $X(\mbT_\mbf/\Q_{\cyc})$ defined in \eqref{X-eq}.

\begin{Th} \label{main-thm}
The product $L_p\bigl(\Sym^2\mbf,\psi,\kappa,\sigma\bigr)\cdot L_p(\psi\epsilon_{\mbf},s-k+1)$ lies in the characteristic ideal of $e_\sigma X(\mbT_\mbf/\Q_{\cyc})$.
\end{Th}

\begin{proof} The divisibility in part (3) of Theorem~\ref{to complete} produces a bound on the characteristic ideal of $e_\eta H^1_{\Gr,2}(\Q_\infty,\mbA)^\vee$. However, this bound applies to the parity where $\eta(-1) = \psi(-1)$, which is the half of the weight space where Dasgupta's factorization formula fails. 

The algebraic functional equation of Theorem~\ref{func-thm} (interchanging the $+1$ and $-1$ eigenspaces via the involution $\iota$) acts on the Iwasawa algebra by mapping the $e_\eta$-component to the dual $e_\sigma$-component. Applying this to our Kolyvagin bound moves the characteristic ideal divisibility from the $e_\eta$-branch to the $e_\sigma$-branch, where $\psi(-1)=-\sigma(-1)$. On this half of the weight space, we apply Perrin-Riou's big dual exponential map $\mathcal{L}$ (see \S\ref{PR-subsubsec}) to the bottom class of our Euler system. Using the explicit reciprocity law for Beilinson--Flach classes in Hida families (\cite[Theorems~8.2.8 and 10.2.2]{KLZ17}), the image of this class under $\mathcal{L}$ coincides with the three-variable Rankin--Selberg $p$-adic $L$-function $L_p(\mbf,\mbf,\psi,1+\mbj)$ from \S\ref{3-subsubsec}, up to the standard Euler factors and periods.

Crucially, since now we have $\psi(-1)=-\sigma(-1)$, Dasgupta's factorization holds. Namely, as recalled in \eqref{L-fact-eq}, we may factor the Rankin--Selberg $p$-adic $L$-function as
\[
L_p(\mbf,\mbf,\psi,1+\mbj)=L_p(\Sym^2\mbf,\psi,\kappa,\sigma)\cdot L_p(\psi\epsilon_{\mbf},s-k+1),
\]
where the second term on the right is the corresponding three-variable Kubota--Leopoldt $p$-adic $L$-function. Since the characteristic ideal of $e_\sigma X(\mbT_\mbf/\Q_{\cyc})$ bounds the index of the projected Euler system and the regulator maps this Euler system to the Rankin-Selberg $p$-adic $L$-function, we obtain an inclusion of ideals
\[
\bigl(L_p(\Sym^2\mbf,\psi,\kappa,\sigma)\cdot L_p(\psi\epsilon_\mbf,s-k+1)\bigr)\subset \charac\bigl(e_\sigma X(\mbT_\mbf/\Q_{\cyc})\bigr),
\]
as was to be shown. \end{proof}

Theorem~\ref{main-thm}, which is Theorem~\ref{C-intro} in the introduction, is a generalization to Hida families of \cite[Theorem~5.4.2]{LZ19}.

\begin{Remark}
The divisibility in Theorem~\ref{main-thm} is non-optimal due to the presence of the parasitic factor $L_p(\psi\epsilon_{\mbf},s-k+1)$: this is a reflection of a well-known limitation of the Kolyvagin system machinery (\emph{cf.} \cite[Remark 5.4.3]{LZ19}). Since the full Rankin--Selberg representation is reducible modulo $p$, one cannot apply the Mazur--Rubin method (\cite{MR_KS}) to it directly. Thus, we are forced to project the Euler system to the (irreducible) symmetric square summand; this restricts our algebraic bound to $X(\mbT_\mbf/\Q_{\cyc})$, whereas the explicit reciprocity law yields the full, unfactored Rankin--Selberg $p$-adic $L$-function.    
\end{Remark}

\bibliographystyle{amsplain}
\bibliography{references}
\end{document}